\documentclass[11pt, a4paper, reqno]{amsart}
\usepackage{amsmath, amsthm, amssymb, url, mathtools}
\usepackage{multirow, pdflscape, tikz}
\usetikzlibrary{arrows}
\usepackage[utf8]{inputenc}
\renewcommand{\arraystretch}{1.5}

\newcommand{\ad}{\mathrm{ad}}
\newcommand{\supp}{\mathrm{supp}}
\newcommand{\Aut}{\mathrm{Aut}}

\newcommand{\Spec}{\mathrm{Spec}}

\newcommand{\proj}{\mathrm{proj}}
\newcommand{\gr}{\mathrm{gr}}

\newcommand{\0}{{\bf 0}}
\newcommand{\1}{{\bf 1}}

\theoremstyle{plain}
\newtheorem{theorem}{Theorem}[section]
\newtheorem{theorem*}{Theorem}
\newtheorem{corollary*}{Corollary}

\newtheorem{corollary}[theorem]{Corollary}
\newtheorem{lemma}[theorem]{Lemma}
\newtheorem{proposition}[theorem]{Proposition}
\newtheorem*{claim*}{Claim}

\theoremstyle{definition}
\newtheorem{definition}[theorem]{Definition}
\newtheorem*{notation}{Notation}

\newtheorem{example}[theorem]{Example}
\newtheorem{remark}[theorem]{Remark}

\renewcommand{\phi}{\varphi}
\renewcommand{\epsilon}{\varepsilon}

\title{Code algebras which are axial algebras and their $\mathbb{Z}_2$-gradings}

\author{Alonso Castillo-Ramirez}
\thanks{Departamento de Matematicas, Centro Universitario de Ciencias Exactas e Ingenierias, Universidad de Guadalajara, Mexico, email: alonso.castillor@academicos.udg.mx}

\author{Justin McInroy}
\thanks{School of Mathematics, University of Bristol, Bristol, BS8 1TW, UK, and the Heilbronn Institute for Mathematical Research, Bristol, UK, email: justin.mcinroy@bristol.ac.uk}

\begin{document}
\maketitle

\begin{abstract}
A code algebra $A_C$ is a non-associative commutative algebra defined via a binary linear code $C$. We study certain idempotents in code algebras, which we call \emph{small idempotents}, that are determined by a single non-zero codeword.  For a general code $C$, we show that small idempotents are primitive and semisimple and we calculate their fusion law.  If $C$ is a projective code generated by a conjugacy class of codewords, we show that $A_C$ is generated by small idempotents and so is, in fact, an axial algebra. Furthermore, we classify when the fusion law is $\mathbb{Z}_2$-graded.  In doing so, we exhibit an infinite family of $\mathbb{Z}_2 \times \mathbb{Z}_2$-graded axial algebras - these are the first known examples of axial algebras with a non-trivial grading other than a $\mathbb{Z}_2$-grading.
\end{abstract}

\section{Introduction}

Both code algebras and axial algebras provide a way of axiomatising important features of vertex operator algebras (VOAs). These were first considered by physicists in connection with 2D conformal field theory, but also later by mathematicians.  The most famous example is the Moonshine VOA $V^\natural$, which has the Monster simple sporadic group as its automorphism group and was instrumental in Borcherd's proof of monstrous moonshine.

Code algebras are a new class of commutative non-associative algebras introduced in \cite{codealgebras}. They are an axiomatisation of code VOAs, a class of VOAs where the representation theory is governed by two linear codes.  Moreover, in every framed VOA $V$, such as $V^\natural$, there exists a unique code sub VOA $W$ and $V$ is a simple current extension of $W$ \cite{codesubVOA}.

Given a binary linear code $C$ of length $n$, a \emph{code algebra} $A_C(\Lambda)$ is a commutative non-associative algebra over a field $\mathbb{F}$ with basis
\begin{align*}
t_i & \qquad i = 1, \dots, n\\
e^\alpha & \qquad \alpha \in C^* := C \setminus \{ \0, \1\}
\end{align*}
where $\Lambda := \left\{ a_{i,\alpha}, b_{\alpha,\beta}, c_{i,\alpha} \in \mathbb{F}  : i \in \supp(\alpha), \alpha, \beta \in C^* , \beta \neq \alpha, \alpha^c \right\}$ is a set of \emph{structure parameters} that determine the products $t_i \cdot e^{\alpha}$, $e^\alpha \cdot e^\beta$ and $e^\alpha \cdot e^\alpha$, respectively. One particularly nice choice of structure parameters is where $a = a_{i, \alpha}$, $b = b_{\alpha, \beta}$ and $c= c_{i, \alpha}$ for all $i \in \supp(\alpha)$, $\alpha, \beta \in C^*$. Roughly speaking, the $t_i$ represent the support of the code, the $e^\alpha$ represent the codewords and the multiplication reflects this. For further details see Definition \ref{CodeAlgebra}.

In this paper, we explore when code algebras are also axial algebras and classify when these have a particularly symmetric multiplicative structure, namely that the fusion law is $\mathbb{Z}_2$-graded. \emph{Axial algebras} are a new class of commutative non-associative algebras that has attracted considerable interest recently (see \cite{Axial2, HSS17, axialstructure, MC18, MR17, R15, R16, S18}) since its introduction by Hall, Rehren and Shpectorov in \cite{Axial1}. The class includes several interesting algebras, in particular, subalgebras of the Greiss algebra, Majorana algebras, Jordan algebras and Matsuo algebras. The defining feature of an \emph{axial algebra} is that it is generated by $\mathcal{F}$-axes.  These are primitive semisimple idempotents which satisfy the fusion law $\mathcal{F}$. More explicitly, the adjoint action of an $\mathcal{F}$-axis $a$ on the algebra decomposes it into a direct sum of eigenspaces
\[
A = \bigoplus_{\lambda \in \mathcal{F}} A_\lambda
\]
where $A_\lambda$ is the $\lambda$-eigenspace, $A_1$ is $1$-dimensional and elements of the eigenspaces multiply according to the fusion law $\mathcal{F}$ (see Definition \ref{axialalgebra} for details).

To show that a code algebra $A_C$ is an axial algebra, we must identify enough primitive, semisimple idempotents to generate the algebra and show that they all satisfy the same fusion law $\mathcal{F}$.  One way to find idempotents is to use the $s$-map construction introduced in \cite[Proposition 5.2]{codealgebras}.

Given a linear subcode $D$ of $C$ and a vector $v \in \mathbb{F}_2^n$
\[
s(D, v) := \lambda \sum_{i \in \supp(D)} t_i + \mu \sum_{\alpha \in D} (-1)^{(\alpha, v)} e^\alpha
\]
is an idempotent of $A_C$, where $\lambda$ and $\mu$ satisfy a linear and quadratic equation (see Proposition \ref{smap}).  In particular, when $D = \{ \0, \alpha \}$, for some $\alpha \in C^*$ and the characteristic of $\mathbb{F}$ is not $2$, or dividing $|\alpha|$, the $s$-map construction gives us two idempotents, which we call \emph{small idempotents}:
\[
e_\pm := \lambda t_\alpha \pm \mu e^\alpha
\]
where $t_\alpha = \sum_{i \in \supp(\alpha)} t_i$.

In \cite{codealgebras}, the eigenvalues, eigenvectors and fusion law were calculated for the small idempotents in the case where $C$ is a constant weight code, that is all non-constant codewords have the same weight.  In this paper, we remove this restriction.  We show that the resulting eigenvalues are $1$, $0$, $\lambda$, $\lambda - \frac{1}{2}$ and $\nu^p_\pm$, for $p=(m, |\alpha|-m)$ that correspond to partitions of the weight of $\alpha$.  We give explicit vectors which form a basis of each eigenspace (see Table \ref{tab:esp}).  In particular, the $1$-eigenspace is $1$-dimensional and the algebra decomposes as a sum of eigenspaces, so $e_\pm$ is a primitive semisimple idempotent.  Furthermore, we calculate its fusion law $\mathcal{F}$, as given in Table \ref{tab:small}.  This allows us to prove the following theorem.

\begin{theorem*}\label{intro:axialalg}
Let $C$ be a projective code and $\alpha \in C$ such that the set $S := \{ \alpha_1, \dots, \alpha_l \}$ of conjugates of $\alpha$ under the action of $\Aut(C)$ generates $C$.  Suppose that the characteristic of $\mathbb{F}$ is not $2$, or dividing $|\alpha|$, and the structure parameters $\Lambda$ satisfy Theorem $\ref{codeisaxialalg}$. Then, the non-degenerate code algebra $A_C(\Lambda)$ is an axial algebra generated by the small idempotents corresponding to the codewords in $S$.
\end{theorem*}

For some codes $C$ and special values of the structure parameters, the fusion law may have a $\mathbb{Z}_2$-grading.  If this is the case, for each axis $a$, we get a decomposition $A = A_+ \oplus A_-$.  Moreover, we may then define an algebra automorphism $\tau_a$ given by the linear extension of
\[
v \mapsto \begin{cases}
v & \mbox{if } v \in A_+ \\
-v & \mbox{if } v \in A_-
\end{cases}
\]
The group generated by the set of all $\tau_a$, for each $\mathcal{F}$-axis $a$, is called the \emph{Miyamoto group}. Hence, such graded fusion laws are of particular interest. For the code axial algebras given by Theorem \ref{intro:axialalg}, we classify when their fusion law is $\mathbb{Z}_2$-graded.

\begin{theorem*}\label{introZ2}
Let $A_C$ be a code algebra satisfying the assumptions of Theorem $\ref{intro:axialalg}$.  Then the fusion law of the small idempotents is $\mathbb{Z}_2$-graded if and only if
\begin{enumerate}
\item[$1.$] $|\alpha| = 1$, $C = \mathbb{F}_2^n$
\begin{enumerate}
\item[\textup{(}a\textup{)}] $n = 2$, $a=-1$.
\item[\textup{(}b\textup{)}] $n=3$.
\end{enumerate}

\item[$2.$] $|\alpha| = 2$, and $C = \bigoplus C_i$ is the direct sum of even weight codes all of the same length $m \geq 3$. 

\item[$3.$] $|\alpha| > 2$, and  $D := \proj_\alpha(C)$ has a codimension one linear subcode $D_+$ which is the union of weight sets of $D$ and $\1 := (1, \dots, 1) \in D_+$.

In this case, we have
\begin{align*}
A_+ &= A_1 \oplus A_0 \oplus A_\lambda \oplus A_{\lambda - \frac{1}{2}} \oplus
\bigoplus_{m \in wt(D_+)} A_{\nu^{(m, |\alpha|-m)}_\pm} \\
A_- &= \bigoplus_{m \in wt(D) - wt(D_+)} A_{\nu^{(m, |\alpha|-m)}_\pm}
\end{align*}
\end{enumerate}
\end{theorem*}

The explicit code algebras and fusion laws obtained in cases 1 and 2 are given in Sections \ref{al=1} and \ref{al=2}, respectively.  Moreover, for some special values of structure parameters in case 2, we get an infinite family of $\mathbb{Z}_2 \times \mathbb{Z}_2$-graded axial algebras.

\begin{theorem*}
If the structure parameters satisfy the conditions in Proposition \ref{Z2xZ2}, then Example $(2)$ in Theorem $\ref{introZ2}$ is a $\mathbb{Z}_2 \times \mathbb{Z}_2$-graded axial algebra.
\end{theorem*}

These are the first known examples of axial algebras with  a non-trivial grading other than a $\mathbb{Z}_2$-grading.

The structure of the paper is as follows.  In Section \ref{sec:background}, we introduce code algebras and axial algebras and review all the relevant preliminary results we will need. The eigenvalues and eigenvectors of small idempotents are calculated in Section \ref{sec:eigen}, hence showing that small idempotents are primitive and semisimple.  Section \ref{sec:fusion} deals with their fusion law.  In Section \ref{sec:examples}, we prove Theorem \ref{intro:axialalg} and give some examples of code algebras which are axial algebras.  In particular, we do the examples where $|\alpha|$ is $1$, or $2$, which are $\mathbb{Z}_2$-graded, and also the $\mathbb{Z}_2 \times \mathbb{Z}_2$-graded example.  The classification of when the fusion law is $\mathbb{Z}_2$-graded is completed in Section \ref{sec:grading}.

\medskip

We would like to thank the Mexican Academy of Sciences for a grant under the Newton Fund/CONACYT for a visit of the second author to the first author in Guadalajara where the majority of this work was done.

\section{Background}\label{sec:background}

We begin by reviewing some facts about codes and fixing notation, before giving the definition and some brief details about axial and code algebras.

Throughout this paper, all algebras will be commutative and non-associative, by which we mean not necessarily associative.

\subsection{Binary linear codes}

Let $\mathbb{F}_2$ be the field with two elements. Recall that a rank $k$ binary linear code $C$ of length $n$ is a $k$-dimensional subspace of $\mathbb{F}_2^n$.  For any $\alpha = (\alpha_1, \dots, \alpha_n) \in \mathbb{F}_2^n$, denote its support by 
\[
\supp(\alpha) := \{ i = 1, \dots, n: \alpha_i = 1 \},
\]
and its Hamming weight by $\vert \alpha \vert :=  \vert \supp(\alpha) \vert$. The support of the code $C$ itself is defined to be $\supp(C) := \bigcup_{\alpha \in C} \supp(\alpha)$ and the set of weights of the codewords in $C$ is denoted $\mbox{wt}(C) := \{ \vert \alpha \vert : \alpha \in C \}$.

A \emph{weight set} of $C$ is the set
\[
W_w(C) = \{ \alpha \in C : |\alpha| = w \}
\]
of all codewords in $C$ of weight $w$.

Two codes $C$ and $D$ are \emph{equivalent} if there exists $g \in S_n$ such that $C^g = D$, where $S_n$ acts naturally on $C$ by permuting the coordinates of the codewords.  We define the automorphism group of $C$ as $\Aut(C) := \{ g \in S_n : C^g = C \}$.

We write $C^*$ for the non-constant codewords in $C$; that is, all codewords which are not $\0 := (0, \dots, 0)$ or $\1 := (1, \dots, 1)$. If $\1 \in C$, then every $\alpha \in C$ has a complement, denoted by $\alpha^c := \1 + \alpha$.  Conversely, if some $\alpha \in C$ has a complement, then $\1 \in C$ and every codeword in $C$ has a complement.

A \emph{generating matrix} for a rank $k$ binary linear code $C$ of length $n$ is a $k \times n$ matrix $G$ whose rows are a basis of $C$. Note that two codes $C$ and $D$ are equivalent if there is a generating matrix $G$ for $C$ and a permutation matrix $P$ such that $G P$ is a generating matrix for $D$.

Given two codes $C$ and $D$, the \emph{direct sum} $C \oplus D$ is the binary linear code whose generating matrix is given by the block diagonal matrix where the two blocks are generating matrices of $C$ and $D$.  A code is called \emph{indecomposable} if it is not equivalent to the direct sum of two non-trivial binary linear codes.

The \emph{dual code} $C^\perp$ of $C$ is the set of all $v \in \mathbb{F}_2^n$ such that $(v,C) = 0$, where $(\cdot, \cdot)$ is the usual inner product.

\begin{definition}
A binary linear code $C$ is \emph{projective} if the minimum weight of a codeword in $C^\perp$ is at least three.
\end{definition}

Let $M$ be a generating matrix for $C$.  Note that $C^\perp$ has a codeword of weight $1$ if and only if $M$ has a column equal to zero, and $C^\perp$ has a codeword of weight $2$ if and only if two columns of $M$ are equal. Thus, $C$ is projective if and only if $M$ has no column equal to zero and its columns are pairwise distinct.

\begin{lemma}\label{Cproj}
Let $C$ be a binary linear code.  Then $C$ is projective if and only if for all $i = 1, \dots, n$, there exists a set of codewords $S$ such that
\[
\{i\} = \bigcap_{\alpha \in S}  \supp(\alpha)
\]
\end{lemma}
\begin{proof}
Suppose that the above property holds.  Then, for all $i$, there exists a codeword $\alpha \in C$ with $\alpha_i = 1$ and hence $C^\perp$ has no codewords of weight $1$.  Moreover, for all $i \neq j$, there exists $\alpha \in C$ such that $\alpha_i \neq \alpha_j$.  Hence, $C^\perp$ has no codeword of weight $2$ and $C$ is projective.

Conversely, suppose that the above property does not hold for some $i = 1, \dots , n$. Either there does not exist a codeword in $C$ supported on $i$, and hence $C^\perp$ contains a codeword of weight one, or there exists $i \neq j$ such that for every codeword $\alpha \in C$, $\alpha_i = \alpha_j$, and hence $C^\perp$ has a codeword of weight two.  In any case, $C$ is not projective.
\end{proof}

Let $S$ be a subset of $\{1, \dots, n\}$ and denote by $\proj_S \colon C \to \mathbb{F}_2^{n-|S|}$ the usual projection map.  Then, the \emph{projection} $\proj_S(C)$ is a binary linear code.  Note that it is the same as the code formed by puncturing the code at all places in $S^c$.  For $\alpha \in C$, we write $\proj_\alpha$ for $\proj_{\supp(\alpha)}$.  By considering the generating matrices, it is easy to see that, if $C$ is a projective code, then $\proj_S(C)$ is also a projective code.

\subsection{Axial algebras}

In this section, we will review the basic definitions related to axial algebras. For further details, see \cite{Axial1, axialstructure}. Let $\mathbb{F}$ be a field not of characteristic two, $\mathcal{F} \subseteq \mathbb{F}$ a subset, and $\star \colon \mathcal{F} \times \mathcal{F} \to 2^{\mathcal{F}}$ a symmetric map. We call the pair $(\mathcal{F}, \star)$ a \emph{fusion law over $\mathbb{F}$} and a single instance $\lambda \star \mu$ a \emph{fusion rule}. We will extend the operation $\star$ to arbitrary subsets $U, V \subseteq \mathcal{F}$ by $U \star V = \bigcup_{\lambda \in U, \mu \in V} \lambda \star \mu$.

Let $A$ be a non-associative (i.e.\ not-necessarily-associative) commutative algebra over $\mathbb{F}$. For an element $a \in A$, the adjoint endomorphism $\ad_a$ is defined by $\ad_a(v):=av$, $\forall v \in A$. Let $\Spec(a)$ be the set of eigenvalues of $\ad_a$, and for $\lambda \in \Spec(a)$, let $A_\lambda(a)$ be the $\lambda$-eigenspace of $\ad_a$. Where the context is clear, we will write $A_\lambda$ for $A_\lambda(a)$.

\begin{definition}\label{axialalgebra}
Let $(\mathcal{F}, \star)$ be a fusion law over $\mathbb{F}$. An element $a \in A$ is an \emph{$\mathcal{F}$-axis} if the following hold:
\begin{enumerate}
\item $a$ is \emph{idempotent} (i.e.\ $a^2 = a$),
\item $a$ is \emph{semisimple} (i.e.\ the adjoint $\ad_a$ is diagonalisable),
\item $a$ is \emph{primitive} (i.e.\ $A_1$ is the linear span of $a$),
\item $\Spec(a) \subseteq \mathcal{F}$ and $A_\lambda A_\mu \subseteq \bigoplus_{\gamma \in \lambda \star \mu } A_{\gamma}$, for all $\lambda, \mu \in \Spec(a)$. 
\end{enumerate}
\end{definition}

\begin{definition}
A non-associative commutative algebra is an \emph{$\mathcal{F}$-axial algebra} if it is generated by $\mathcal{F}$-axes.
\end{definition}

When the fusion law is clear from context we drop the $\mathcal{F}$ and simply use the term \emph{axial algebra}. The Monster fusion law is given by:

\begin{table}[!htb]
\setlength{\tabcolsep}{4pt}
\renewcommand{\arraystretch}{1.5}
\centering
\begin{tabular}{c||c|c|c|c}
 & $1$ & $0$ & $\frac{1}{4}$ & $\frac{1}{32}$ \\ \hline \hline
$1$ & $1$ &  & $\frac{1}{4}$ & $\frac{1}{32}$ \\ \hline
$0$ &  & $0$ &$\frac{1}{4}$ & $\frac{1}{32}$ \\ \hline
$\frac{1}{4}$ & $\frac{1}{4}$ & $\frac{1}{4}$ & $1, 0$ & $\frac{1}{32}$ \\ \hline
$\frac{1}{32}$ & $\frac{1}{32}$  & $\frac{1}{32}$ & $\frac{1}{32}$ & $1, 0, \frac{1}{4}$ 
\end{tabular}
\vspace{3pt}
\caption{Monster fusion law}
\end{table}

\noindent and are exhibited by the $2A$-axes in the Griess algebra.  A \emph{Majorana algebra} is an axial algebra with the Monster fusion law which also satisfies some additional axioms (see \cite{I09} for details). These kinds of algebra generalise subalgebras of the Griess algebra.

\begin{definition}
The fusion law $\mathcal{F}$ is \emph{$T$-graded}, where $T$ is a finite abelian group, if there exist a partition $\{ \mathcal{F}_t : t \in T \}$ of $\mathcal{F}$ such that for all $s,t \in T$,
\[
\mathcal{F}_s \star \mathcal{F}_t \subseteq \mathcal{F}_{st}
\]
\end{definition}
We allow the possibility that some part $\mathcal{F}_t$ is the empty set. Let $A$ be an algebra and $a \in A$ an $\mathcal{F}$-axis (we do not require $A$ to be an axial algebra).  If $\mathcal{F}$ is \emph{$T$-graded}, then the axis $a$ defines a \emph{$T$-grading} on $A$ where the $t$-graded subspace $A_t$ of $A$ is
\[
A_t = \bigoplus_{\lambda \in \mathcal{F}_t } A_\lambda(a)
\]
When $\mathcal{F}$ is $T$-graded we may define some automorphisms of the algebra.  Let $T^*$ denote the linear characters of $T$.  That is, the homomorphisms from $T$ to $\mathbb{F}^\times$. For an axis $a$ and $\chi\in T^*$, consider the linear map $\tau_a(\chi)\colon A\to A$ defined by 
\[
u \mapsto \chi(t) u \qquad \mbox{for } u \in A_t(a)
\]
and extended linearly to $A$. Since $A$ is $T$-graded, this map $\tau_a(\chi)$ is an automorphism of $A$. Furthermore, the map sending $\chi$ to $\tau_a(\chi)$ is a homomorphism from $T^*$ to $\Aut(A)$.

The subgroup $T_a := \langle \tau_a(\chi) : \chi \in T^* \rangle$ is called the \emph{axial subgroup} corresponding to $a$.  For a set $S$ of $\mathcal{F}$-axes, the subgroup $\langle T_a : a \in S \rangle$ of $\Aut(A)$ is called the \emph{Miyamoto group with respect to $S$}.  When $A$ is an axial algebra and $S$ is its generating set of axes, we write $G := \mbox{Miy}(A)$ for the Miyamoto group.

We are particularly interested in $\mathbb{Z}_2$-graded fusion laws. In this case, we identify $\mathbb{Z}_2$ with the group $\{ +, - \}$ equipped with the usual multiplication of signs. For example, the Monster fusion law $\mathcal{F}$ is $\mathbb{Z}_2$-graded where $\mathcal{F}_+ = \{ 1,0,\frac{1}{4} \}$ and $\mathcal{F}_- = \{ \frac{1}{32} \}$. 

When the fusion law is $\mathbb{Z}_2$-graded and $\mathrm{char}(\mathbb{F}) \neq 2$, then $T^* = \{ \chi_1, \chi_{-1}\}$, where $\chi_1$ is the trivial character on $T = \mathbb{Z}_2$ and $\chi_{-1}$ is the sign character. Here, the axial subgroup contains just one non-trivial automorphism, $\tau_a(\chi_{-1})$.  We write this as $\tau_a \colon A \to A$ and call it the \emph{Miyamoto involution associated to $a$}.  It is defined by the linear extension of
\[
v \tau_a = \begin{cases}
v & \mbox{if } v \in A_+ \\
-v & \mbox{if } v \in A_-
\end{cases}
\]

\subsection{Code algebras}

We define code algebras as non-associative algebras that generalise some properties of code VOAs.

\begin{definition}\label{CodeAlgebra}
Let $C \subseteq \mathbb{F}_2^n$ be a binary linear code of length $n$, $\mathbb{F}$ a field and $\Lambda \subseteq \mathbb{F}$ be a collection of structure parameters
\[
\Lambda := \left\{ a_{i,\alpha}, b_{\alpha,\beta}, c_{i,\alpha} \in \mathbb{F}  : i = 1, \dots, n, \alpha, \beta \in C^* \right\}.
\]
The \emph{code algebra} $A_C(\Lambda)$ is the commutative algebra over $\mathbb{F}$ with basis
\[
\{ t_i : i = 1, \dots, n \} \cup \{ e^{\alpha} : \alpha \in C^* \},
\]
and multiplication given by
\begin{align*}
t_i \cdot t_j & = \delta_{i,j} t_i \\
t_i \cdot e^\alpha & = \begin{cases} 
a_{i,\alpha} \, e^\alpha & \text{if } \alpha_i = 1 \\
\mathrlap0\phantom{ \sum \limits_{i \in \supp(\alpha) }c_{i,\alpha} t_i} & \text{if } \alpha_i =0
\end{cases} \\
e^\alpha \cdot e^\beta & = \begin{cases}
b_{\alpha, \beta}\, e^{\alpha + \beta} & \text{if } \alpha \neq \beta, \beta^c \\
 \sum \limits_{i \in \supp(\alpha) }c_{i,\alpha} t_i & \text{if } \alpha = \beta  \\
0 & \text{if } \alpha = \beta^c
\end{cases}
\end{align*}
\end{definition}

We say that a code algebra $A_C$ is \emph{non-degenerate} if $\supp(C) = \{1, \dots, n\}$, $|C^*| >0$ and all the structure parameters in $\Lambda$ are non-zero.  In this paper, we will always assume code algebras are non-degenerate.  We will call the basis elements $t_i$ \emph{toral elements} and the $e^\alpha$ \emph{codewords elements}.

A code algebra $A_C$ has some obvious idempotents $t_i$.  We can also construct additional idempotents using the $s$-map construction.  We say that a code $D$ has \emph{constant weight} if all non-constant codewords have the same weight. That is, all codewords in $D^* = D \setminus \{\0, \1\}$ have the same weight. Suppose that $D$ is a linear subcode of $C$ of constant weight.  By \cite[Lemma 4.1]{codealgebras}, we know that for any $\beta \in D^*$, the number of pairs $(\alpha, \gamma) \in D^* \times D^*$ such that $\beta = \alpha + \gamma$ is 
\[
 l = 2|D^*| - |D|. 
\]
Define $t_D := \sum_{i \in \supp(D)} t_i$.

We now give a slightly altered version of the $s$-map construction that appears in \cite[Proposition 4.2]{codealgebras}.  We briefly reprove it while paying close attention to the characteristic of the field.  Since we assume the code algebra is non-degenerate, the structure constants are non-zero. However, if $\mathbb{F}$ has characteristic $p \neq 0$, then $d$, which is the weight of a codeword in $D^*$, or $l$, is zero in $\mathbb{F}$ when $p \mid d$, or $p \mid l$, respectively. Note also that we are only interested in solutions where $\mu \neq 0$, since otherwise we just have a sum of toral idempotents.

\begin{proposition}\label{smap}
Let $\mathbb{F}$ be a field of characteristic not $2$ and $C$ a binary linear code of length $n$.  Suppose that $D$ is a constant weight subcode of $C$, where the weight of a codeword in $D^*$ is $d$, and the structure parameters supported on $D^*$ are constant $(a,b,c)$. Take $v \in \mathbb{F}_{2}^{n}$ and $\mu, \lambda \in \mathbb{F}$ with $\mu \neq 0$. Then, there exists an idempotent in $A_C$ of the form
\[
s(D,v) := \lambda t_D + \mu \sum_{\alpha \in D^*}\left( -1\right) ^{\left( v ,\alpha \right) }e^{\alpha },
\]
if and only if
\begin{equation}\label{eq1}
1 = 2ad \lambda + bl \mu
\end{equation}
and either 
\begin{enumerate}
\item[(a)] $d$ is not zero in $\mathbb{F}$ and $\mu$ satisfies the equation
\[
\left[ b^2 l^2+2a^{2}cd^2(l+2)\right] \mu^{2} + 2b l (ad-1) \,\mu + 1-2ad=0
\]
\end{enumerate}
or,
\begin{enumerate}
\item[(b)]$l$ is not zero in $\mathbb{F}$ and $\lambda$ satisfies the equation
\[
2b^2 l^2\,\lambda^{2} - 2\left[ b^2l^2 + acd(l+2) \right] \lambda + c(l+2) =0.
\]
\end{enumerate}
\end{proposition}
\begin{proof}
As in the proof of \cite[Proposition 4.2]{codealgebras}, the condition $s(D,v) = s(D,v)^2$ gives us two equations
\begin{align}
\lambda & = \lambda^2+\mu^2c \frac{l+2}{2}, \label{quad3} \\
\mu & =  2\lambda \mu ad + \mu^{2}b l.  \label{lin}
\end{align}
(This differs slightly from \cite[Proposition 4.2]{codealgebras}, as $\frac{l+2}{2} = |D^*|\frac{d}{m}$, where $m = |\supp(D)|$.)

Since $\mu \neq 0$, we may divide Equation \ref{lin} by $\mu$ to get Equation \ref{eq1}.  We see from this that at most one of $d$ and $l$ can be zero in $\mathbb{F}$.  If $d$ is not zero, then we may substitute for $\lambda$ in Equation \ref{quad3} and obtain the quadratic in (a).  Similarly, if $l$ is not zero in $\mathbb{F}$, we substitute for $\mu$ in Equation \ref{quad3} to obtain the quadratic in (b).
\end{proof}

It is clear that we may always extend the field $\mathbb{F}$ to $\mathbb{F}(r)$, where $r$ is a root of one of the above quadratics. Then, for the algebra over this extended field, $s(D,v)$ will exist.

Fix $\alpha \in C^*$.  The subcode spanned by $\alpha$, $D = \langle \alpha \rangle$, is a constant weight subcode. We assume that the structure parameters supported on $D^*$ are constant, namely $a_\alpha := a_{i, \alpha} = a_{j, \alpha}$ and $c_\alpha := c_{i, \alpha} = c_{j, \alpha}$ for $i \in \supp(\alpha)$. We have $l = 0$, so we must also assume that $d = |\alpha|$ is not zero in $\mathbb{F}$. So, by Proposition \ref{smap}, the following are idempotents
\[
e_{\pm} := \lambda t_\alpha \pm \mu e^{\alpha},
\]
where the equations for $\lambda$ and $\mu$ are simplified to
\[
\lambda := \frac{1}{2a_\alpha \vert \alpha \vert} \ \text{ and } \ \mu^2 : = \frac{\lambda - \lambda^2}{c_{\alpha}}.
\]
We call these \emph{small idempotents}.  In \cite{codealgebras}, their eigenvalues, eigenvectors and fusion law were calculated in the case where $C$ itself was a constant weight code.  This paper generalises those results to an arbitrary code $C$.


\section{Eigenvalues and eigenvectors}\label{sec:eigen}

In this section, we will calculate the eigenvalues and eigenvectors of a small idempotent $e_\pm$, show that they span the whole algebra and therefore that $e_\pm$ is semisimple.  Throughout this section we will fix a binary linear code $C$, a codeword $\alpha \in C^*$ and $\mathbb{F}$ will have characteristic $p \neq 2$ such that $p \nmid |\alpha|$ (we allow $p=0$). Let $e = e_+$ be the small idempotent defined by the $s$-map.  We begin by defining some notation.

\begin{notation}
Throughout the paper, we write statements involving $\1 \in C$, or the complement $\alpha^c$ of a codeword $\alpha$.  We do not assume that $\1 \in C$, or complements exist, just that if they do, then these statements should hold.
\end{notation}

For $\beta, \gamma \in C$, we use the notation $\beta \cap \gamma := \supp(\beta) \cap \supp(\gamma)$.

\begin{definition}
Given $\beta \in C^*$, we define the \emph{weight partition} to be the unordered pair
\[
p(\beta) :=\left( | \alpha \cap \beta|, | \alpha \cap (\alpha + \beta)| \right) = ( |\alpha \cap \beta|, | \alpha \cap \beta^c| )
\]
Note that $p(\beta) = p(\alpha+\beta) = p(\beta^c)$.  Let
\[
C_\alpha(p) := \{ \beta \in C^*\setminus \{\alpha, \alpha^c\} : p(\beta) = p \}
\]
be the set of all $\beta$ which give the weight partition $p$.  We define
\[
P_\alpha := \{ p(\beta) : \beta \in C^* \setminus \{ \alpha, \alpha^c\} \}
\]
to be the set of all weight partitions of $\alpha$.
\end{definition}

We make the following assumptions on the structure parameters:
\begin{align*}
a  &:= a_{i, \beta} && \mbox{for all } i \in \supp(\beta), \beta \in C^*\\
b_{\alpha, \beta} &= b_{\alpha, \gamma} && \mbox{for all } \beta,\gamma \in C_\alpha(p), p \in P_\alpha\\
c_\alpha &:= c_{i, \alpha} && \mbox{for all } i \in \supp(\alpha)
\end{align*}
In other words, the $a$ structure parameter is the same for the whole algebra, while the $c$ structure parameter depends on the codeword and the $b$ structure parameter for $\alpha$ depends on the weight set.

In order to give the eigenvectors, we first need to define some scalars which will be their coefficients.  For $\beta \in C^* \setminus \{ \alpha, \alpha^c \}$, we define
\[
\xi_{|\alpha\cap\beta|} := \frac{\lambda a}{2\mu b_{\alpha, \beta}}(|\alpha| - 2|\alpha \cap \beta|) = \frac{1 }{4 \mu b_{\alpha, \beta}} \left(1 - \frac{2 \vert \alpha \cap \beta\vert}{\vert \alpha \vert } \right)
\]
and let $\theta^\beta_\pm$ be the two roots of
\[
x^2 + 2\xi_\beta x -1 = 0.
\]
If $\mathbb{F}$ do not contain these roots, we replace $\mathbb{F}$ by $\mathbb{F}(\theta^\beta_\pm)$.

Where $\alpha$ is understood, to simplify notation, we will write $\xi_\beta := \xi_{|\alpha \cap \beta|}$, $\theta^\beta_\pm := \theta^{|\alpha \cap \beta|}_\pm$.  We observe that $\xi_{|\alpha\cap\beta|}$ depends only on the size of the intersection of $\beta$ with $\alpha$, not on the codeword $\beta$ itself.  

\begin{lemma}\label{coefsubs}
Let $\beta, \gamma, \delta \in C^* \setminus \{ \alpha, \alpha^c \}$ such that $|\alpha \cap \beta| = |\alpha \cap \gamma|$ and $| \alpha \cap \delta| = |\alpha| - |\alpha \cap \beta|$.
\begin{enumerate}
\item[$1.$] $\xi_\beta= \xi_\gamma$
\item[$2.$] $\xi_{\delta}=  -\xi_{\beta}$, in particular $\xi_{\alpha+\beta} = \xi_{\beta^c} = -\xi_\beta$
\item[$3.$] $\theta^\beta_\pm = \theta^\gamma_\pm$
\item[$4.$] $\theta^{\delta}_\pm = -\theta^{\beta}_\mp$, in particular $\theta^{\alpha + \beta}_\pm = \theta^{\beta^c}_\pm  = -\theta^\beta_\mp$
\item[$5.$] $\theta^\beta_+  + \theta^\beta_- = -2\xi_\beta$ and $\theta^\beta_+ \theta^\beta_- = -1$
\item[$6.$] $\frac{1}{\theta^\beta_\pm} = - \theta^\beta_\mp$
\item[$7.$] If $\mathbb{F}$ has characteristic $0$, then
\[
\theta^\beta_\pm =  - \xi_\beta \pm \sqrt{ (\xi_\beta)^2 + 1 }
\]
\end{enumerate}
\end{lemma}
\begin{proof}
By our assumptions on the $b$ structure parameters and using the observation that $\alpha \cap \beta^c = \alpha \cap (\alpha + \beta)$, the first five parts are clear.  The sixth follows from the fifth and the seventh follows from solving the quadratic.
\end{proof}

We also note the following result which will be useful, particularly in positive characteristic.

\begin{lemma}\label{xi0}
We have that $\theta^\beta_+ = -\theta^\beta_-$ if and only if $\theta^\beta_+ = \pm1$ if and only if $\xi_\beta = 0$ if and only if $|\alpha| - 2|\alpha \cap \beta|$ is zero in $\mathbb{F}$.  In characteristic $0$, this only happens when $|\alpha \cap \beta| = \frac{|\alpha|}{2}$.
\end{lemma}

For $p \in P_\alpha$, $\beta \in P_\alpha(p)$, we define
\[
\nu^p_\pm := \frac{1}{4} + \mu b_{\alpha, \beta} (\theta^\beta_\pm + \xi_\beta)
\]
which will turn out to be an eigenvalue.  We note that $\nu^p_\pm$ is well-defined.  Indeed, parts one and three of Lemma \ref{coefsubs} show that $\theta^\beta_\pm + \xi_\beta = \theta^\gamma_\pm + \xi_\gamma$ if $|\alpha \cap \beta| = |\alpha \cap \gamma|$.  If $| \alpha \cap \delta| = |\alpha| - |\alpha \cap \beta|$, then $\theta^\delta_\pm + \xi_\delta = -\theta^\beta_\mp - \xi_\beta = \theta^\beta_\pm + 2\xi_\beta - \xi_\beta = \theta^\beta_\pm + \xi_\beta$ by parts two, four and five. So, by our assumptions on $b_{\alpha, \beta}$, $\nu^p_\pm$ is constant for all $\beta \in C_\alpha(p)$.

Let $p \in P_\alpha$ be a weight partition and $\beta \in C_\alpha(p)$.  We define
\[
w_\pm^\beta:= \theta^\beta_\pm e^\beta + e^{\alpha + \beta}
\]
which will be an eigenvector for $\nu^p_\pm$.

\begin{lemma}\label{nuspace}
Let $p \in P_\alpha$ and $\beta \in C_\alpha(p)$.  Then,
\[
w^{\alpha + \beta}_\pm = -\theta^\beta_\mp w^\beta_\pm
\]
\end{lemma}
\begin{proof}
By Lemma \ref{coefsubs}, we have
\begin{align*}
w^{\alpha+\beta}_\pm & = \theta^{\alpha + \beta}_\pm e^{\alpha+\beta} + e^\beta \\
&= -\theta^\beta_\mp e^{\alpha + \beta} + e^\beta \\
&= - \theta^\beta_\mp(-\tfrac{1}{\theta^\beta_\mp}e^\beta + e^{\alpha+\beta}) = -\theta^\beta_\mp w^\beta_\pm \qedhere
\end{align*}
\end{proof}

Since $\beta$ and $\alpha+\beta$ define the same eigenvector up to scaling, we pick a subset $C_\alpha'(p)$ of $C_\alpha(p)$ such that for every $\beta \in C_\alpha(p)$, either $\beta \in C_\alpha'(p)$, or $\alpha + \beta \in C_\alpha'(p)$, but not both.  We may now list the eigenvectors for $e$ and show that they form a basis for their eigenspaces.

From now on, we assume that the field $\mathbb{F}$ over which $A_C$ is defined contains the roots of $x^2 + 2\xi_\beta x -1 = 0$, for all $\beta \in C^* \setminus \{ \alpha , \alpha^c\}$.  In other words, the $\theta^\beta_\pm$ are defined.

\begin{proposition}\label{espaces}
Suppose $a \neq \frac{1}{2 \vert \alpha \vert}, \frac{1}{3 \vert \alpha \vert}$ and $\theta^\beta_+ \neq \theta^\beta_-$ for all $\beta \in C^* \setminus \{ \alpha, \alpha^c\}$.  The sets of eigenvectors for $e = e_+$ given in Table $\ref{tab:esp}$ are a basis for their eigenspace.  Moreover, $A$ decomposes as a direct sum of these eigenspaces, hence $e$ is semisimple. It is primitive if $\nu^p_\pm \neq 1$ for all $p \in P_\alpha$.

\begin{table}[!htb]
\setlength{\tabcolsep}{4pt}
\renewcommand{\arraystretch}{1.5}
\centering
\begin{tabular}{l|l@{\hskip 15pt}l}
Eigenvalue & Eigenvector \\
\hline
$1$ & $e = \lambda t_\alpha + \mu e^\alpha$  \\
\hline
\multirow{2}{4em}{$0$} &  $t_i$ & for $i \not \in \supp(\alpha)$ \\
& $e^{\alpha^c}$ \\
\hline
$\lambda$ & $t_j - t_k$ & for $k \in \supp(\alpha)$, $k \neq j$ \\
\hline
$\lambda - \tfrac{1}{2}$ & $2\mu c_\alpha t_\alpha - e^\alpha$ \\
\hline
$\nu^p_\pm$ & $w_\pm^\beta = \theta^\beta_\pm e^\beta + e^{\alpha + \beta}$ & for $\beta \in C_\alpha'(p)$, $p \in P_\alpha$
\end{tabular}
\vspace{4pt}
\\where $j \in \supp(\alpha)$ is fixed
\caption{Eigenspaces for small idempotents}\label{tab:esp}
\end{table}
\end{proposition}

Note that $\theta^\beta_+ \neq \theta^\beta_-$ is equivalent to the quadratic $x^2 + 2\xi_\beta x -1=0$ not having a repeated root. This proposition will be proven via the two following lemmas.

\begin{lemma}
The vectors listed in Table $\ref{tab:esp}$ are eigenvectors for the given eigenvalues.
\end{lemma}
\begin{proof}
It is clear that $e$ is a $1$-eigenvector because it is an idempotent. Observe that, for $i \not \in \supp(\alpha)$,
\[
(\lambda t_\alpha + \mu e^{\alpha}) \cdot t_i = 0
\]
and
\[
(\lambda t_\alpha + \mu e^{\alpha})  \cdot e^{\alpha^c}= 0
\]
Now, for $i, j \in \supp(\alpha)$, we have
\[
(\lambda t_\alpha + \mu e^{\alpha})   \cdot (t_i - t_j) = \lambda t_i + \mu a e^\alpha  - \lambda t_j - \mu a e^\alpha = \lambda ( t_i - t_j)
\]
Also,
\begin{align*}
  (\lambda t_\alpha + \mu e^{\alpha})  \cdot (2\mu c_\alpha t_\alpha - e^\alpha) &= (2 \lambda \mu c_\alpha - \mu c_\alpha )t_\alpha  - (\lambda a \vert \alpha \vert - 2 \mu^2 c_\alpha a \vert \alpha \vert) e^\alpha   \\
& = (\lambda - \tfrac{1}{2} ) 2 \mu c_\alpha t_\alpha  - ( \lambda a \vert \alpha \vert - 2 (\lambda - \lambda^2 ) a \vert \alpha \vert ) e^\alpha \\
& =  (\lambda - \tfrac{1}{2} ) 2 \mu c_\alpha t_\alpha  - (\tfrac{1}{2} - (1-\lambda)) e^\alpha \\
& =  (\lambda - \tfrac{1}{2} ) 2 \mu c_\alpha t_\alpha  - ( \lambda - \tfrac{1}{2} ) e^\alpha
\end{align*}

Now consider the element $x e^\beta + e^{\alpha + \beta}$, for $\beta \in C^* \setminus \{ \alpha, \alpha^c \}$ and some $x \in \mathbb{F}^\times$.  By our assumptions on the structure parameters, we have:
\begin{align*} 
(\lambda t_\alpha + \mu e^{\alpha}) \cdot  (x  e^\beta + e^{\alpha + \beta}) 
&  = \big(\lambda x a \vert \alpha \cap \beta \vert  + \mu b_{\alpha,\beta}\big)e^\beta \\
& \qquad + \big(\lambda a \vert \alpha \cap (\alpha + \beta) \vert + \mu x b_{\alpha, \beta}\big) e^{\alpha + \beta}
\end{align*}
This element $x e^\beta + e^{\alpha + \beta}$ is a $\nu$-eigenvector, for some $\nu \in \mathbb{F}$, if and only if we have the following:
\begin{align*}
x\nu &= \lambda x a \vert \alpha \cap \beta \vert  + \mu b_{\alpha,\beta} \\
\nu &= \lambda a \vert \alpha \cap (\alpha + \beta) \vert + \mu x b_{\alpha, \beta}
\end{align*}

We can now easily check that $\nu = \nu^p_\pm$ and $x = \theta^\beta_\pm$ is a solution for these.  We do this for the second equation only. By rearranging the definition of $\xi_{\alpha\cap \beta}$, we get that $\lambda a |\alpha \cap (\alpha+\beta)| = \frac{1}{2}\lambda a |\alpha| - \mu b_{\alpha, \beta} \xi_{\alpha\cap\beta} = \frac{1}{4} - \mu b_{\alpha, \beta} \xi_{\alpha\cap\beta}$.  We substitute to find
\begin{align*}
\nu & = \lambda a | \alpha \cap (\alpha + \beta)| + \mu b_{\alpha, \beta} \theta^\beta_\pm \\
& = \tfrac{1}{4} + \mu b_{\alpha, \beta}(\theta^\beta_\pm -  \xi_{\alpha\cap\beta}) \\
& = \tfrac{1}{4} + \mu b_{\alpha, \beta}(\theta^\beta_\pm +  \xi_\beta)\\
&= \nu^p_\pm \qedhere
\end{align*}
\end{proof}

\begin{lemma}\label{basis}
Suppose $a \neq \frac{1}{3 \vert \alpha \vert}$ and $\theta^\beta_+ \neq \theta^\beta_-$ for all $\beta \in C^* \setminus \{ \alpha, \alpha^c\}$.  Then the eigenvectors listed in Table $\ref{tab:esp}$ are a basis for $A$.
\end{lemma}
\begin{proof}
Suppose that $\1 \in C$, the proof for $\1\notin C$ is similar.  Let $\mathcal{B}$ be the set of eigenvectors listed in Table \ref{tab:esp}.  Counting we have the following: 
\[
\vert \mathcal{B} \vert = 1 + (n - \vert \alpha \vert) +1+ (\vert \alpha \vert - 1) + 1 +  2 \cdot (\frac{\vert C^* \vert - 2}{2})  = n + \vert C^* \vert = \dim(A)
\]
In order to show that $\mathcal{B}$ is linearly independent, we shall write the matrix $M$ consisting of the elements of $\mathcal{B}$ (in a slightly different order to the one given above and with one element scaled) with respect to the ordered basis
\begin{align*}
\{t_j\} &\cup \{t_k : k \in \supp(\alpha), k \neq j\} \cup \{e^\alpha, e^{\alpha^c}\} \cup \{t_i : i \not \in \supp(\alpha)\} \\
& \qquad  \cup \{e^{\beta}, e^{\alpha + \beta} : \beta \in C_\alpha'(p), p \in P_\alpha \}
\end{align*}

We have
\[
M = \left( \begin{array}{c|c} \begin{matrix}
\lambda & \lambda & \dots & \lambda & \mu \\
1 & 1 & \dots & 1 & -\frac{1}{2 \mu c_\alpha}\\
1 & -1 & & & \\
& & \ddots & &  \\
1 & \dots & & -1 & \\
\end{matrix} & \\
\hline
& \begin{matrix}
1 & & & & & & & & \\
 & 1 & & & & & & & \\
 & & \ddots & & & & & & \\
& & & 1 & & & & & \\
& & & & \theta_+^{\beta_1} & 1 & & & \\
& & & & \theta_-^{\beta_2} & 1 & & & \\
& & & & & & \ddots & & \\
& & & & & & & \theta_+^{\beta_r} & 1 \\
& & & & & & & \theta_-^{\beta_r} & 1
\end{matrix}
\end{array} \right)
\]
where the matrix has $0$ in all the blank spaces. As $\theta^\beta_+ \neq \theta^\beta_-$ and both are non-zero for all $\beta \in C_\alpha'(p)$, $p \in P_\alpha$, we know that $\det(M) \neq 0$ if and only if the determinant of the top left block $M'$ is nonzero.  After using row operations to simplify the first two rows, we see that
\begin{align*}
\det(M') &= \mu \vert \alpha \vert (-1)^{\vert \alpha \vert -1} + \tfrac{1}{2 \mu c_\alpha} \lambda \vert \alpha \vert (-1)^{\vert \alpha \vert -1} \\
&= |\alpha|(-1)^{\vert \alpha \vert -1}\left( \mu + \tfrac{\lambda}{2 \mu c_\alpha}  \right)
\end{align*}
This is zero if and only if $0 = \lambda + 2 \mu^2 c_\alpha = \lambda + 2(\lambda - \lambda^2)= \lambda(3 - 2 \lambda)$ and hence $\lambda = \frac{3}{2}$ which is equivalent to $a = \frac{1}{3 \vert \alpha \vert}$.  
\end{proof}

\begin{remark}
We note that $a = \frac{1}{2 \vert \alpha \vert}$ and $a = \frac{1}{3 \vert \alpha \vert}$, correspond to $\lambda = 1$ and $\lambda = \frac{3}{2}$, respectively.  The first of these implies that $\mu = 0$ and hence the $s$ map idempotent just becomes a sum of $t_i$.  The second would collapse the $\lambda - \frac{1}{2}$ eigenspace into the $1$-eigenspace.  Since we only wish to consider the case when $e_\pm$ different from the toral idempotents and primitive, from now on we will rule out these two values for $a$.
\end{remark}

\section{The fusion law}\label{sec:fusion}

We now calculate the fusion law $\mathcal{F} = (\mathcal{F}, \star)$ for the small idempotent $e = e_+$.  Since $A_C$ is commutative, it suffices to calculate just the upper half of $\mathcal{F}$.  Note that we already know the row for $1$, as $e$ is primitive and the values of $\mathcal{F}$ are eigenvalues for $e$. 

We restate our previous assumptions, making further assumptions on the $b$ and $c$ structure parameters.
\begin{align*}
a  &:= a_{i, \beta} && \mbox{for all } i \in \supp(\beta), \beta \in C^*\\
b_{\alpha, \beta} &= b_{\alpha, \gamma} && \mbox{for all } \beta,\gamma \in C_\alpha(p), p \in P_\alpha \\
b_{\alpha^c, \beta} &= b_{\alpha^c, \gamma} &&\mbox{for all } \beta,\gamma \in C_\alpha(p), p \in P_\alpha \\
c_\beta &:= c_{i, \beta} && \mbox{for all } i \in \supp(\beta), \beta \in C^*
\end{align*}
So, the $a$ structure parameter is the same for the whole algebra, while the $c$ structure parameter depends on the codeword and the $b$ structure parameter for $\alpha$ and $\alpha^c$ depends on the weight sets.

Recall that we also assume that $a \neq \frac{1}{2|\alpha|}, \frac{1}{3|\alpha|}$, $\theta^\beta_+ \neq \theta^\beta_-$ for all $\beta \in C^* \setminus \{ \alpha, \alpha^c\}$ and the characteristic of $\mathbb{F}$ is not $2$, or dividing $|\alpha|$.

\begin{theorem}
The fusion law for the above small idempotent $e$ is given in Table $\ref{tab:small}$, where $P_\alpha = \{ p_1, \dots p_k\}$.

\begin{table}[!htb]
\setlength{\tabcolsep}{4pt}
\renewcommand{\arraystretch}{1.5}
\centering
\begin{tabular}{c|ccccccc}
 & $1$ & $0$ & $\lambda$ & $\lambda-\frac{1}{2}$ & $\nu^{p_1}_\pm$ &  $\dots$ & $\nu^{p_k}_\pm$ \\ \hline
$1$ & $1$ &  & $\lambda$ & $\lambda-\frac{1}{2}$ & $\nu^{p_1}_\pm$ &  $\dots$ & $\nu^{p_k}_\pm$ \\
$0$ &  & $0$ &  &  & $\nu^{p_1}_\pm$ &  $\dots$ & $\nu^{p_k}_\pm$ \\
$\lambda$ & $\lambda$ &  & $1, \lambda, \lambda-\frac{1}{2}$ &  & $\nu^{p_1}_+, \nu^{p_1}_-$  & \dots & $\nu^{p_k}_+, \nu^{p_k}_-$ \\
$\lambda-\frac{1}{2}$ & $\lambda-\frac{1}{2}$&  &  & $1, \lambda-\frac{1}{2}$ & $\nu^{p_1}_+, \nu^{p_1}_-$ & \dots & $\nu^{p_k}_+, \nu^{p_k}_-$ \\
$\nu^{p_1}_\pm$ & $\nu^{p_1}_\pm$ & $\nu^{p_1}_\pm$  & $\nu^{p_1}_+, \nu^{p_1}_-$ & $\nu^{p_1}_+, \nu^{p_1}_-$  & $X_1$ && $N(p_1, p_k)$ \\
$\vdots$ & $\vdots$ &$\vdots$ &$\vdots$ &$\vdots$ & & $\ddots$ &\\
$\nu^{p_k}_\pm$ & $\nu^{p_k}_\pm$ & $\nu^{p_k}_\pm$  & $\nu^{p_k}_+, \nu^{p_k}_-$ & $\nu^{p_k}_+, \nu^{p_k}_-$ & $N(p_k, p_1)$ & & $X_k$
\end{tabular}
\vspace{5pt}

where
\[
N(p,q) := \{ \nu^{p(\beta + \gamma)}_+, \nu^{p(\beta + \gamma)}_- : \beta \in C_\alpha'(p), \gamma \in C_\alpha'(q), \gamma \neq \beta, \alpha+\beta, \beta^c, \alpha+\beta^c \}
\]

\vspace{5pt}
and $X_i$ represents the table

\vspace{5pt}
\begin{tabular}{c|cc}
& $\nu^{p_i}_+$ & $\nu^{p_i}_-$ \\
\hline
$\nu^{p_i}_+$ & $1,0,\lambda, \lambda-\frac{1}{2}, N(p_i, p_i)$ & $1,0,\lambda, \lambda-\frac{1}{2}, N(p_i, p_i)$ \\
$\nu^{p_i}_-$ & $1,0,\lambda, \lambda-\frac{1}{2}, N(p_i, p_i)$ & $1,0,\lambda, \lambda-\frac{1}{2}, N(p_i, p_i)$
\end{tabular}

\caption{Fusion law for small idempotents}\label{tab:small}
\end{table}

\end{theorem}

\begin{remark}
Note that entries of the fusion table could sometimes be replaced by subsets of the entry given due to either some intersection properties of the code, or special values of some coefficients.  Some of these special cases will be useful for us later.  For these we will explicitly give a case analysis of when the answer can be a subset of the generic answer.  Where we do not carry out such an analysis the answer is labelled as `generic', which means that answers which are subsets may still be possible.

If eigenvalues coincide, then the corresponding columns and rows of the fusion table would merge, which might cause a problem in our goal of classifying $\mathbb{Z}_2$-graded fusion tables, or cause the small idempotent to be non-primitive. However, we could use a slightly more general definition of $\mathcal{F}$-axis where we have a map $\lambda \colon \mathcal{F} \to \Spec(\ad_a)$ from labels $f \in \mathcal{F}$ to eigenvalues and we require that an axis satisfy $A_{f_1} A_{f_2} \subseteq A_{f_1 \star f_2}$.  Since we do not require the map $\lambda$ to be injective, this allows us to split the eigenspace and treat differently each part.  In particular, we can then drop the requirement for an axis to be primitive.
\end{remark}

The theorem will be proved via a series of calculations. Throughout, let $p \in P_\alpha$, $\beta \in C_\alpha'(p)$.  

\subsection*{Calculation of $0\star \_$}

The $0$-eigenspace has a basis $t_i$ such that $i \notin \supp(\alpha)$ and also $e^{\alpha^c}$ if $\1 \in C$.

\begin{lemma}\label{0*0}
$0 \star 0 = 0$.  In particular, $0 \star 0 \neq \emptyset$.
\end{lemma}
\begin{proof}
We have $t_i t_j = \delta_{ij} t_i$ and $t_i e^{\alpha^c} = a e^{\alpha^c}$.
\end{proof}

\begin{lemma}
$0 \star \lambda = \emptyset$
\end{lemma}
\begin{proof}
Let $i \notin \supp(\alpha)$ and $j,k \in \supp(\alpha)$.  Then $t_i (t_j -t_k) = 0$ and, by our assumptions on the $a$ structure parameters, $e^{\alpha^c}(t_j -t_k) = a(e^{\alpha^c} - e^{\alpha^c}) = 0$.
\end{proof}

\begin{lemma}
$0 \star \lambda-\frac{1}{2} = \emptyset$
\end{lemma}
\begin{proof}
We have $t_i (2\mu c_\alpha t_\alpha - e^\alpha) = 0$ and $e^{\alpha^c}(2\mu c_\alpha t_\alpha - e^\alpha) = 0$.
\end{proof}

\begin{lemma}\label{0*nu}
We have
\[
0 \star \nu^p_\pm = \begin{cases}
\emptyset & \mbox{if } \1 \notin C \mbox{ and for all } \beta \in C_\alpha(p), \alpha \cap \beta = \beta \\
\nu^p_\pm & \mbox{otherwise}
\end{cases}
\]
\end{lemma}
\begin{proof}
Let $i \notin \supp(\alpha)$.  Then $i \in \supp(\beta)$ if and only if $i \in \supp(\alpha+\beta)$.
\[
t_i(\theta^\beta_\pm e^\beta + e^{\alpha+\beta}) = \begin{cases}
0 & \text{if } i \notin \supp(\beta) \\
a(\theta^\beta_\pm e^\beta + e^{\alpha+\beta}) & \text{if } i \in \supp(\beta)
\end{cases} \\
\]

If $\1 \in C$, then we must also consider $e^{\alpha^c}$.  Since $b_{\alpha^c, \beta} = b_{\alpha^c, \alpha+\beta}$ and, by Lemma \ref{coefsubs}, $\theta^\beta_\pm = \theta^{\alpha+\beta^c}$, we have
\begin{align*}
e^{\alpha^c}(\theta^\beta_\pm e^\beta + e^{\alpha+\beta}) &= b_{\alpha^c, \beta} \theta^\beta_\pm e^{\alpha^c + \beta} + b_{\alpha^c, \alpha+\beta} e^{\beta^c}\\
&= b_{\alpha^c, \beta}(\theta^{\alpha^c+\beta}_\pm e^{\alpha+\beta^c} + e^{\beta^c}) 
\end{align*}
By Lemma \ref{nuspace} this is also in the $\nu^p_\pm$-eigenspace.
\end{proof}

\subsection*{Calculation of $\lambda \star \_$}

Fixing $i \in \supp(\alpha)$, the $\lambda$-eigenspace is spanned by $t_i - t_j$ where $j \in \supp(\alpha) \setminus \{ i\}$.  Note that the $\lambda$-eigenspace only exists if $|\alpha| >1$.

\begin{lemma}\label{lambda*lambda}
We have
\[
\lambda \star \lambda = \begin{cases}
1, \lambda - \frac{1}{2} & \mbox{if } |\alpha| =2 \\
1, \lambda, \lambda - \frac{1}{2} & \mbox{otherwise}
\end{cases}
\]
\end{lemma}
\begin{proof}
We have
\[
(t_i - t_j)(t_i - t_k) = t_i + \delta_{jk} t_j
\]
The eigenspace is spanned by just one vector, $t_i - t_j$, if and only if $|\alpha| =2$.  Then, the product $t_i + t_j \in A_1 \oplus A_{\lambda - \frac{1}{2}}$.  However, otherwise we get the product $t_i \in A_1 \oplus A_\lambda \oplus A_{\lambda - \frac{1}{2}}$.
\end{proof}

\begin{lemma}
$\lambda \star \lambda - \frac{1}{2} = \emptyset$
\end{lemma}
\begin{proof}
Since $i,j \in \supp(\alpha)$, $(t_i - t_j)(2\mu c_\alpha t_\alpha - e^\alpha) = 0$.
\end{proof}

\begin{lemma}\label{lambda*nu}
We have
\[
\lambda \star \nu^p_\pm = \begin{cases}
\emptyset & \mbox{if } p = ( 0, |\alpha|) \\
\nu^p_\mp & \mbox{if } \xi_\beta = 0 \mbox{ where } p(\beta) = p \\
\nu^p_+, \nu^p_- & \mbox{otherwise}
\end{cases}
\]
\end{lemma}
\begin{proof}
Note that $i,j \in \supp(\alpha)$.  We get three cases:
\[
(t_i - t_j)(\theta^\beta_\pm e^\beta + e^{\alpha+\beta}) = \begin{cases}
0 & \text{if } i,j \in \supp(\beta) \\
0 & \text{if } i,j \notin \supp(\beta) \\
a(\theta^\beta_\pm e^\beta - e^{\alpha+\beta}) & \text{if } |\{i, j\} \cap \supp(\beta)| = 1
\end{cases}
\]
The third case never occurs if and only if we always have $\alpha \cap \beta = \0$, or $\alpha$, which is equivalent to $p = ( 0, |\alpha|)$.  Suppose this is not the case.  Generically, the third case is in $A_{\nu^p_+} \oplus A_{\nu^p_-}$.  However, it is in $A_{\nu^p_\mp}$ if and only if $\theta^\beta_\pm = -\theta^\beta_\mp$ which, by Lemma \ref{xi0}, is if and only if $\xi_\beta=0$. Note that, since $\xi^\beta \neq 0$, $\theta^\beta_\pm = -\theta^\beta_\pm$ is impossible, and hence the result cannot be in $A_{\nu^p_\pm}$.
\end{proof}

\subsection*{Calculation of $\lambda - \frac{1}{2} \star \_$}

\begin{lemma}\label{lambda-1/2*lambda-1/2}
We have
\[
\lambda - \tfrac{1}{2} \star \lambda - \tfrac{1}{2} = \begin{cases}
1 & \mbox{if } a = -\frac{1}{|\alpha|} \\
\lambda - \tfrac{1}{2} & \mbox{if } a = \frac{1}{|\alpha|} \\
1, \lambda - \tfrac{1}{2} & \mbox{otherwise}
\end{cases}
\]
\end{lemma}
\begin{proof}
\[
(2\mu c_\alpha t_\alpha - e^\alpha)(2\mu c_\alpha t_\alpha - e^\alpha) = (4 \mu^2 c_\alpha^2 + c_\alpha) t_\alpha -4\mu c_\alpha |\alpha|a e^\alpha
\]
Generically this is in $A_1 \oplus A_{\lambda - \frac{1}{2}}$ and it also cannot be zero.

The result is in $A_{\lambda - \frac{1}{2}}$ if and only if for some $\zeta \in \mathbb{F}$,
\begin{align*}
\zeta(2\mu c_\alpha) &= 4\mu^2 c_\alpha^2 + c_\alpha \\
-\zeta &= -4\mu c_\alpha |\alpha| a
\end{align*}
We eliminate the $\zeta$ and substitute $\mu^2 = \frac{\lambda-\lambda^2}{c_\alpha}$ to get an equation in $\lambda$:
\[
(1-\lambda)^2 = \tfrac{1}{4}
\]
Recall that we do not allow $\lambda = \frac{3}{2}$.  The remaining solution $\lambda = \frac{1}{2}$ is equivalent to $a = \frac{1}{|\alpha|}$.

Finally, the result is in $A_1$ if and only if for some $\zeta \in \mathbb{F}$,
\begin{align*}
\zeta \lambda&= 4\mu^2 c_\alpha^2 + c_\alpha \\
\zeta \mu &= -4\mu c_\alpha |\alpha| a
\end{align*}
Since $\mu \neq 0$, we may divide the second equation by $\mu$ and substitute into the first to again eliminate $\zeta$.  Again, we substitute for $\mu^2$ to get
\[
4\lambda^2-4\lambda -3 = 0
\]
which has two solutions $\frac{3}{2}$ and $-\frac{1}{2}$.  As above, the first of these is not allowed and the second is equivalent to $a = -\frac{1}{|\alpha|}$.
\end{proof}

\begin{lemma}\label{lambda-1/2*nu}
Generically, $\lambda - \frac{1}{2} \star \nu^p_\pm = \nu^p_+, \nu^p_-$.  However, $\lambda - \frac{1}{2} \star \nu^p_\pm = \nu^p_\pm$ if and only if either $\xi_\beta = 0$, or $a = \frac{1}{|\alpha|}$.
\end{lemma}
\begin{proof}
We have
\begin{align*}
(2\mu c_\alpha t_\alpha - e^\alpha)(\theta^\beta_\pm e^\beta + e^{\alpha+\beta}) 
&= (2\mu c_\alpha \theta^\beta_\pm a |\alpha \cap \beta| - b_{\alpha, \alpha+\beta})e^\beta \\
& \qquad + (2 \mu c_\alpha a | \alpha \cap (\alpha + \beta)| - b_{\alpha, \beta}\theta^\beta_\pm) e^{\alpha+\beta}
\end{align*}
which is generically in $A_{\nu^p_+} \oplus A_{\nu^p_-}$.

Since $b_{\alpha, \alpha+ \beta} = b_{\alpha, \beta}$, the above is in $\nu^p_\pm$ if and only if
\[
2\mu c_\alpha \theta^\beta_\pm a |\alpha \cap \beta| - b_{\alpha, \beta} = \theta^\beta_\pm \left(2 \mu c_\alpha a | \alpha \cap (\alpha + \beta)| - b_{\alpha, \beta}\theta^\beta_\pm \right)
\]
We divide by $\theta^\beta_\pm$ and use Lemma \ref{coefsubs} to obtain
\[
2\mu c_\alpha a |\alpha \cap \beta| + b_{\alpha, \beta} \theta^\beta_\mp = 2 \mu c_\alpha a | \alpha \cap (\alpha + \beta)| - b_{\alpha, \beta}\theta^\beta_\pm
\]
Rearrange to get
\begin{align*}
b_{\alpha, \beta} \left(\theta^\beta_\pm + \theta^\beta_\mp \right) &= 2\mu c_\alpha a \left(| \alpha \cap (\alpha + \beta)| - |\alpha \cap \beta|\right) \\
-2b_{\alpha, \beta} \xi_\beta &= 2\mu c_\alpha a \left(|\alpha| -2 |\alpha \cap \beta| \right) \\
&= \frac{2\mu^2b_{\alpha, \beta} c_\alpha}{\lambda} \xi_\beta \\
&= 2(1-\lambda) b_{\alpha, \beta}\xi_\beta 
\end{align*}
Cancelling the $b_{\alpha, \beta}$ and rearranging once more we find that
\[
0 = \xi_\beta \left(2\lambda -1\right)
\]
and $\lambda = \frac{1}{2}$ is equivalent to $a= \frac{1}{|\alpha|}$, completing the proof.
\end{proof}

\subsection*{Calculation of $\nu^p_\pm \star \_$}

We begin by performing calculating the products of the basis elements here as these calculations are needed for finding the fusion law, but will also be useful elsewhere.

\begin{lemma}\label{nu*nucalc}
Let $\beta, \gamma \in C$ such that $\beta \neq \alpha, \alpha^c$, $\gamma \neq \beta, \beta^c, \alpha+\beta, \alpha + \beta^c$ and $\epsilon, \iota = \pm$.
\begin{enumerate}
\item[$1.$] $w^\beta_\epsilon w^\beta_\iota = \theta^\beta_\epsilon \theta^\beta_\iota c_\beta t_\beta + c_{\alpha+\beta} t_{\alpha+\beta} + b_{\beta, \alpha+\beta}(\theta^\beta_\epsilon + \theta^\beta_\iota) e^\alpha$ which is generically in $A_1 \oplus A_0 \oplus A_\lambda \in A_{\lambda-\frac{1}{2}}$

\item[$2.$] $w^\beta_\epsilon w^{\beta^c}_\iota =(b_{\beta, \alpha+\beta^c}\theta^\beta_\epsilon - b_{\beta^c, \alpha+\beta}\theta^\beta_{-\iota}) e^{\alpha^c} \in A_0$

\item[$3.$] $w^\beta_\epsilon w^\gamma_\iota = (\theta^\beta_\epsilon \theta^\gamma_\iota b_{\beta,\gamma} + b_{\alpha+\beta, \alpha+\gamma})e^{\beta+\gamma} + (\theta^\beta_\epsilon b_{\beta, \alpha + \gamma} + \theta^\gamma_\iota b_{\alpha+\beta, \gamma})e^{\alpha+\beta+\gamma}$, which is generically in $A_{\nu^{p(\beta+\gamma)}_+} \oplus A_{\nu^{p(\beta+\gamma)}_-}$
\end{enumerate}
\end{lemma}
\begin{proof}
These are straightforward calculations.
\end{proof}

\begin{lemma}\label{nu*nu}
Let $p,q \in P_\alpha$ be two different weight partitions of $\alpha$ and $\epsilon, \iota = \pm$.  Generically we have the following:
\begin{enumerate}
\item[$1.$] $\nu^p_\epsilon \star \nu^q_\iota = N(p,q)$,
\item[$2.$] $\nu^p_\epsilon \star \nu^p_\iota = N(p,p) \cup \{1,0,\lambda, \lambda - \frac{1}{2} \}$.
\end{enumerate}
where
\begin{align*}
N(p,q) := \{ \nu^{p(\beta + \gamma)}_+, \nu^{p(\beta + \gamma)}_- &: \beta \in C_\alpha'(p), \gamma \in C_\alpha'(q), \\
& \qquad \gamma \neq \beta, \beta^c, \alpha+\beta, \alpha+\beta^c \}.
\end{align*}

However, if $b_{\beta, \alpha+\beta^c} = b_{\beta^c, \alpha+\beta}$ and $c_\beta = c_{\alpha + \beta}$ for all $\beta \in C_\alpha(p)$, then for $p = ( \frac{|\alpha|}{2}, \frac{|\alpha|}{2} )$, 
\[
\nu^p_\epsilon \star \nu^p_\iota = \begin{cases}
N(p,p) \cup \{1,0,\lambda- \frac{1}{2} \} & \mbox{if } \epsilon = \iota \\
N(p,p) \cup \{\lambda \} & \mbox{if } \epsilon = -\iota
\end{cases}
\]
\end{lemma}
\begin{proof}
To begin, let $\gamma \neq \beta, \beta^c, \alpha+\beta, \alpha+\beta^c$.  By Lemma \ref{nu*nucalc}, we have $N(p,q) \subseteq \nu^p_\epsilon \star \nu^q_\iota$ generically.  

We note that the remaining cases for $\gamma$ all have $p(\gamma) = p(\beta)$. Now, by Lemma \ref{nuspace}, $w^{\alpha+\beta}_\pm$ is a scalar multiple of $w^\beta_\pm$, and $w^{\alpha+\beta^c}_\pm$ is a scalar multiple of $w^{\beta^c}_\pm$.  So we are left with two cases: $\gamma = \beta$ and $\gamma = \beta^c$.  Again by Lemma \ref{nu*nucalc}, these are generically in $\{1,0,\lambda, \lambda - \frac{1}{2} \}$.

Now suppose that the conditions on the structure parameters hold and $p = (\frac{|\alpha|}{2}, \frac{|\alpha|}{2} )$. So, by Lemma \ref{xi0}, we may assume $\theta^\beta_\epsilon = \epsilon$. Consider $w^\beta_\epsilon w^\beta_\iota$ from part (1) of Lemma \ref{nu*nucalc}.  Note that $t_\beta = t_{\alpha \cap \beta} + t_{\alpha^c \cap \beta}$ and similarly for $t_{\alpha+\beta}$. So, as $\alpha^c \cap \beta = \alpha^c \cap (\alpha+\beta)$, if $\epsilon = \iota$, the coefficients of the $t_i$ for $i \in \supp(\alpha)$ are all equal.  Hence, the product is in $A_1 \oplus A_0 \oplus A_{\lambda - \frac{1}{2}}$.

Similarly, if $\epsilon = -\iota$, then, for all $i \notin \supp(\alpha)$, the $t_i$ and $e^\alpha$ terms cancel and we see that $w^\beta_{+}w^\beta_{-}$ is in $A_\lambda$.  Again, by the assumptions on the structure parameters and part (2) of Lemma \ref{nu*nucalc}, we see that $w^\beta_{+}w^{\beta^c}_{-} = 0$, therefore the result follows.
\end{proof}

\section{Axial algebras and examples}\label{sec:examples}

We wish to generate our non-degenerate code algebra $A_C$ by idempotents and hence show that it is an axial algebra.  In order to do this, we consider the small idempotents obtained from a set $S = \{ \alpha_1, \dots, \alpha_l\}$, where each $\alpha_i$ is a conjugate of $\alpha$ under $\Aut(C)$.  Note that, since the $\alpha_j$ are conjugate, the weight sets $P_{\alpha_j}$ and $P_{\alpha_k}$ are equal for $j,k = 1, \dots l$.

\begin{theorem}\label{codeisaxialalg}
Let $C$ be a projective code and $\alpha \in C$ such that the set $S = \{ \alpha_1, \dots, \alpha_l\}$ of conjugates of $\alpha$ under $\Aut(C)$ generates the code.  Suppose that the characteristic of $\mathbb{F}$ is not $2$, or dividing $|\alpha|$, and the structure parameters $\Lambda = \{ a_{i, \beta}, b_{\beta, \gamma}, c_{i, \beta}\}$ are such that
\begin{align*}
a  &:= a_{i, \beta} &&\mbox{for all } i \in \supp(\beta), \beta \in C^*\\
b_{\alpha_j, \beta} &= b_{\alpha_k, \gamma} &&\mbox{for all } \beta \in C_{\alpha_j}(p), \gamma \in C_{\alpha_k}(p), p \in P_\alpha \\
b_{\alpha_j^c, \beta} &= b_{\alpha_k^c, \gamma} &&\mbox{for all } \beta \in C_{\alpha_j}(p), \gamma \in C_{\alpha_k}(p), p \in P_\alpha \\
c_\beta &:= c_{i, \beta} &&\mbox{for all } i \in \supp(\beta), \beta \in C^*
\end{align*}

Then, the code algebra $A_C(\Lambda)$ is an axial algebra with respect to the small idempotents and has fusion law given in Table $\ref{tab:small}$.
\end{theorem}
\begin{proof}
We have two small idempotents $e_\pm := \lambda t_\alpha \pm \mu e^\alpha$ defined by $\alpha$.  By subtracting the two, we obtain (a scalar multiple of) the codeword element $e^\alpha$.  The set $S$ generates the code, so by multiplying the $e^\alpha$ where $\alpha \in S$, we can generate all codeword elements of $A_C$.  Since $C$ is projective, by Lemma \ref{Cproj}, for all $i \in 1, \dots, n$ there exists $\beta_1, \dots, \beta_k \in C$ which are pairwise distinct such that
\[
\{ i \} = \bigcap_{j=1}^k \beta_j
\]
Hence $(e^{\beta_1})^2 \dots (e^{\beta_k})^2$ is a scalar multiple of $t_i$.  Since $S$ is a set of conjugates under the automorphism group of the code, the fusion laws for the small idempotents are the same.
\end{proof}

We now give some examples.  Throughout, we assume that $C$ is a projective code, $S$ is a set of conjugates of some $\alpha \in C$ which generate the code and $A_C$ is non-degenerate.

\subsection{$\vert \alpha \vert = 1$}\label{al=1}

If $|\alpha| = 1$ and a set of conjugates $S$ of $\alpha$ generate $C$, then $C$ must be the full vector space $C= \mathbb{F}^n_2$ and $S$ the set of all weight one vectors. It is clear that the only possible weight partition of $\alpha$ is $p = (0,1)$. Moreover, this exists precisely when $n \geq 3$.  Indeed, we may disregard the case $n=1$ as the algebra is degenerate (in fact, in this case, $A_C \cong \mathbb{F}$).  When $n=2$, there are only $\alpha$ and $\alpha^c$.  So, there does not exists $\beta \in C^* \setminus \{ \alpha, \alpha^c \}$ such that $|\alpha \cap \beta| = 0,1$.  For $n \geq 3$, such a $\beta$ does exist. By Proposition \ref{espaces}, the possible eigenvalues of a small idempotent $e$ are $1$, $0$, $\lambda - \frac{1}{2}$, $\nu^p_+$ and $\nu^p_-$ (note that $\lambda$ does not appear as eigenvalue).  

By Theorem \ref{codeisaxialalg} and Table \ref{tab:small}, $A_C$ is an axial algebra with its fusion law given by Table \ref{Fusion-alpha1} when $n \geq 3$.  When $n=2$, the same table applies if we ignore the $\nu^p_\pm$.

\begin{table}[!htb]
\setlength{\tabcolsep}{5pt}
\renewcommand{\arraystretch}{1.5}
\centering
\begin{tabular}{c|cccccc}
 & $1$ & $0$ & $\lambda-\frac{1}{2}$ & $\nu^{p}_+$ & $\nu^p_-$ \\ \hline
$1$ & $1$ &  & $\lambda-\frac{1}{2}$ & $\nu^{p}_+$ & $\nu^p_-$  \\
$0$ &  & $0$ &   & $\nu^{p}_+$ & $\nu^p_-$ \\
$\lambda-\frac{1}{2}$ & $\lambda-\frac{1}{2}$&   & $1, \lambda-\frac{1}{2}$ & $\nu^{p}_+, \nu^{p}_-$ & $\nu^{p}_+, \nu^{p}_-$ \\
$\nu^{p}_+$ & $\nu^{p}_+$ & $\nu^{p}_+$  & $\nu^{p}_+, \nu^{p}_-$ & $1,0,\lambda-\frac{1}{2}, N(p,p)$ & $1,0,\lambda-\frac{1}{2}, N(p,p)$ \\
$\nu^{p}_-$ & $\nu^{p}_-$ & $\nu^{p}_-$  &  $\nu^{p}_+, \nu^{p}_-$ & $1,0,\lambda-\frac{1}{2}, N(p,p)$ & $1,0,\lambda-\frac{1}{2}, N(p,p)$
\end{tabular}
\caption{Fusion law for $\vert \alpha \vert = 1$}
\label{Fusion-alpha1}
\end{table}

We wish to identify when the fusion law is $\mathbb{Z}_2$-graded.  It is easy to see that $1$ and $0$ must both be in the positive part.  We will assume two results which follow from Section \ref{sec:fusion}, but which are proved later in Section \ref{sec:grading}.  Firstly, $\lambda - \frac{1}{2}$ is also in the positive part if $n \neq 2$ and secondly that $\nu^p_+$ and $\nu^p_-$ have the same grading (Lemma \ref{nusame}).  Here we explore whether these cases actually lead to a non-trivial grading.  We begin by analysing the set $N(p,p)$.

\begin{lemma}\label{al=1N}
$N(p,p) = \emptyset$ if and only if $n = 3$.
\end{lemma}
\begin{proof}
When $n = 2$, the weight partition $p$ does not occur.  So, we may assume that $n \geq 3$.  As there is only one partition for $\vert \alpha \vert =1$, we have $N(p,p) = \emptyset$ if and only if there does not exist $\beta, \gamma \in C^* \setminus \{ \alpha, \alpha^c \}$ such that $\gamma \neq \beta, \beta^c, \alpha + \beta, \alpha + \beta^c$. This condition is satisfied if and only if $\vert C^* \setminus \{ \alpha, \alpha^c \} \vert \leq 4$, which happens if and only if $n \leq 3$.  
\end{proof}

\begin{proposition}
The fusion law given by Table $\ref{Fusion-alpha1}$ has a $\mathbb{Z}_2$-grading if and only if $n = 3$, or $n= 2$ and $a = -1$.  In particular, if $n=3$, the $\mathbb{Z}_2$-grading is given by 
\begin{equation}\label{w1-grading1}
A_+ = A_1 \oplus A_0 \oplus A_{\lambda - \frac{1}{2}} \text{  and  } A_- = A_{\nu^p_+} \oplus A_{\nu^p_-}. 
\end{equation}
If $n= 2$ and $a=-1$, we have $\lambda -\frac{1}{2} \star \lambda -\frac{1}{2} = 1$ and the $\mathbb{Z}_2$-grading is given by 
\begin{equation}\label{w1-grading2}
A_+ = A_1 \oplus A_0 \text{  and  } A_- =  A_{\lambda - \frac{1}{2}}.
\end{equation}
\end{proposition}
\begin{proof}
As noted above, it is clear that $1$ and $0$ must be in the positive part.  We also assume that $\lambda - \frac{1}{2}$ is in the positive part if $n \neq 2$ and that $\nu^p_+$ and $\nu^p_-$ have the same grading.

Suppose that $n=2$. Then, the $p=(0,1)$ partition doesn't occur and the only possible grading is when $\lambda - \frac{1}{2}$ is in the negative part.  By Lemma \ref{lambda-1/2*lambda-1/2}, $\lambda - \frac{1}{2} \in \lambda - \frac{1}{2} \star \lambda - \frac{1}{2}$ if and only if $a \neq - \frac{1}{\vert \alpha \vert} = -1$. Hence, we also must have $a = -1$ and the grading is that given in (\ref{w1-grading2}).

For $n = 3$, $N(p,p) = \emptyset$ by Lemma \ref{al=1N}, and it is routine to check that (\ref{w1-grading1}) is a $\mathbb{Z}_2$-grading. 

Finally, if $n \geq 4$, then generically $N(p,p) \neq \emptyset$. However, we must check whether special values of the structure parameters could give $\nu^p_\pm \notin \nu^p_\epsilon \star \nu^p_\iota$ and hence a valid grading.  Assume for a contradiction that they do.  In particular, for all distinct $\beta, \gamma \in C_\alpha(p)$, we must have $w^\beta_+ w^\gamma_- = 0$.  However, since $n \geq 4$, there exists distinct $\beta, \gamma \in C_\alpha(p)$ such that $\gamma \neq \beta, \beta^c, \alpha + \beta, \alpha+ \beta^c$, $|\beta| = |\gamma| = 2$ and $|\alpha \cap \beta| = |\alpha \cap \gamma| = 1$.  From Lemma \ref{nu*nucalc}, we have
\[
w^\beta_+ w^\gamma_- = (\theta^\beta_+ \theta^\gamma_- b_{\beta,\gamma} + b_{\alpha+\beta, \alpha+\gamma})e^{\beta+\gamma} + (\theta^\beta_+ b_{\beta, \alpha + \gamma} + \theta^\gamma_- b_{\alpha+\beta, \gamma})e^{\alpha+\beta+\gamma}
\]
Since we assume this is zero, in particular we require $\theta^\beta_+ b_{\beta, \alpha + \gamma} + \theta^\gamma_- b_{\alpha+\beta, \gamma} = 0$.  However, $\alpha+\beta$ and $\alpha+\gamma$ both have weight $1$ and hence are conjugate to $\alpha$.  Moreover, $\gamma \neq \alpha + \beta, \alpha+ \beta^c$ and so $\gamma \in C_{\alpha+\beta}(p)$; similarly $\beta \in C_{\alpha+\gamma}(p)$.  So, by our assumptions on the structure parameters, $b_{\alpha+\beta, \gamma} = b_{\alpha+\gamma, \beta}$ and we have
\[
0 = b_{\alpha+\beta, \gamma}(\theta^\beta_+ + \theta^\gamma_-) = -2 \xi_\beta b_{\alpha+\beta, \gamma}
\]
since $\theta^\gamma_\pm = \theta^\beta_\pm$.  However, $\xi_\beta \neq 0$, a contradiction.  Hence for $n \geq 4$, there is no non-trivial $\mathbb{Z}_2$-grading.
\end{proof}

\subsection{$C$ is a direct sum of even weight codes}\label{al=2}

Let $C$ be a direct sum of even weight subcodes $C_i$ and $|\alpha| = 2$.  That is, the $C_i$ are the codimension one subcodes of some $\mathbb{F}_2^{m_i}$ which contain all the codewords of even length.  Since we are assuming that $C$ is generated by conjugates of $\alpha$, the the lengths $m_i$ of the $C_i$ must all be the same.  Let this be $m$ and $n = m^r$. Thus,
\[
C = \bigoplus_{i = 1}^r C_i \qquad \mbox{ where } C_i \mbox{ all have length } m
\]
Since we also assume that $C$ is projective, this means that $ n \geq m \geq 3$.  Clearly, the partition $(1,1)$ always exists. The partition $(0,2)$ generally exists, but there are some small degenerate cases in which it does not. Namely, when $n = m = 3, 4$ and the only weight partition of $\alpha$ is $(1,1)$.  Apart from this degenerate case, we have $m \geq 3$ and $n \geq 5$ and exactly two weight partitions, $(0,2)$ and $(1,1)$.  For ease of notation, we will label these by $0=(0,2)$ and $1=(1,1)$

We now consider what the sets $N(p,q)$ are for the different weight partitions $p$ and $q$.

\begin{lemma}\label{al2N1}
Generically, we have
\[
N(1, 1) = \begin{cases}
\emptyset & \mbox{if }  n= m=3, 4 \\
\nu^{0}_+, \nu^{0}_- & n \geq 5
\end{cases}
\]
\end{lemma}
\begin{proof}
 If $n=m=3$, then $|C'_\alpha({1})| = 1$ and so $N({1}, {1}) = \emptyset$.  When $n=m=4$, $|C'_\alpha({1})| = 2$, but the sum of the two distinct codewords in $C'_\alpha({1})$ is $\alpha$, or $\alpha^c$.  Hence again $N({1},{1}) = \emptyset$.  If neither of these cases hold, then $n \geq 5$ and there exist two distinct codewords $\beta, \gamma \in C'_\alpha({1})$ such that $\beta+\gamma \neq \alpha, \alpha^c$.  Their sum $\beta+\gamma$ has weight partition $(0,2)$.
\end{proof}

\begin{lemma}\label{al2N2}
Generically, we have
\begin{align*}
N(0, 0) &= \nu^{0}_+, \nu^{0}_-  \\
N(1, 0) &= \nu^{1}_+, \nu^{1}_- .
\end{align*}
\end{lemma}
\begin{proof}
If the $(0, 2)$ weight partition exists, then $n \geq 5$ and $m \geq 3$ and there exist two distinct codewords $\beta, \gamma \in C'_\alpha({0})$.  Since their sum also has weight partition $(0,2)$, $N({0}, {0}) = \nu^{0}_+, \nu^{0}_-$.  The second claim is clear.
\end{proof}

Now that we know the $N(p,q)$ sets generically, we calculate the fusion rule for the $\nu^p_\pm$.  We do this in a careful way since some choices of the structure parameters will yield a more symmetric table.

\begin{lemma}\label{al2nu0*nu0}
Let $\epsilon, \iota = \pm 1$, we have
\[
\nu^{0}_\epsilon \star \nu^{0}_\iota = 1, 0, \lambda - \tfrac{1}{2}, \nu^{0}_+, \nu^{0}_-
\]
\end{lemma}
\begin{proof}
By Lemmas \ref{nu*nucalc} and \ref{al2N2}, we just need to consider $w^\beta_\epsilon w^\beta_\iota$ for $\beta \in C_\alpha (p)$
\[
(\theta_\epsilon^\beta e^\beta + e^{\alpha + \beta}) (\theta_\iota^\beta e^\beta + e^{\alpha + \beta}) = \theta_\epsilon^\beta \theta_\iota^\beta c_\beta t_\beta + c_{\alpha + \beta} t_{\alpha + \beta} + b_{\beta,\alpha+\beta} (\theta_\epsilon^\beta + \theta_\iota^\beta)e^\alpha
\]
which generically is in $A_1 \oplus A_1 \oplus A_\lambda \oplus A_{\lambda - \frac{1}{2}} $.  However, if $p=(0,2)$, without loss of generality, we may assume that $|\alpha \cap \beta| = 0$.  So, $t_{\alpha+\beta} = t_\alpha + t_\beta$.  Hence the coefficients for each $t_i$ where $i \in \supp(\alpha)$ are the same.  Therefore, the above product is in fact contained in  $A_1 \oplus A_1 \oplus A_{\lambda - \frac{1}{2}}$.
\end{proof}

\begin{lemma}\label{al2nu1*nu1}
Generically, we have
\[
\nu^{1}_\epsilon \star \nu^{1}_\iota = \begin{cases}
1,0,\lambda, \lambda-\frac{1}{2}, \nu^{0}_+, \nu^{0}_- & \mbox{if } n \geq 5 \\
1,0,\lambda, \lambda-\frac{1}{2} & \mbox{if } n=m =3 ,4
\end{cases}
\]
If $b_{\beta, \gamma} = b_{\alpha+\beta, \gamma}$ and $c_\beta = c_{\alpha+\beta}$ for all $\beta, \gamma \in C_\alpha({1})$ then,
\[
\nu^{1}_\epsilon \star \nu^{1}_\iota = \begin{cases}
1,0,\lambda-\frac{1}{2}, \nu^{0}_+, \nu^{0}_- & \mbox{if } \epsilon = \iota \mbox{ and } n \geq 5 \\
1,0,\lambda-\frac{1}{2} & \mbox{if } \epsilon = \iota \mbox{ and } n=m =3,4 \\
\lambda & \mbox{if } \epsilon = -\iota
\end{cases}
\]

\end{lemma}
\begin{proof}
The generic product follows directly from Lemmas \ref{nu*nu} and \ref{al2N1}.  For the special case, using our assumptions on the $b$ structure parameters, observe that
\[
b_{\beta, \alpha+\beta^c} = b_{\alpha+\beta, \alpha+\beta^c} = b_{\alpha+\beta,\beta^c}
\]
So, by the special case of Lemma \ref{nu*nu}, most of the above result follows.  It remains to check $w^\beta_\epsilon w^\gamma_{-\epsilon}$, where $\gamma \neq \beta, \alpha+\beta, \beta^c, \alpha+\beta^c$. Recall that $\theta^\beta_\epsilon = \epsilon$ for $\beta \in C_\alpha({1})$.  By Lemma \ref{nu*nucalc} and our assumptions on the $b_{\beta, \gamma}$,
\begin{align*}
(\epsilon e^\beta + e^{\alpha + \beta}) (-\epsilon e^\gamma + e^{\alpha + \gamma}) &= (- b_{\beta, \gamma} + b_{\alpha + \beta, \alpha+\gamma}) e^{\beta + \gamma} \\
& \qquad + \epsilon (b_{\beta, \alpha+\gamma} - b_{\alpha + \beta, \gamma}) e^{\alpha + \beta + \gamma}\\
&=0\qedhere
\end{align*}
\end{proof}

We have the usual result for $\nu^{1}_{\epsilon} \star \nu^{0}_{\iota}$ generically, but when we make assumptions on the structure parameters, we can get the following.
\begin{lemma}\label{al2nu1*nu0}
If $b_{\beta, \gamma} = b_{\alpha+\beta, \gamma}  = b_{\beta, \alpha+\gamma}$ for all $\beta \in C_\alpha({1})$, $\gamma \in C_\alpha({0})$, then for $\epsilon, \iota = \pm 1$,
\[
\nu^{1}_{\epsilon}  \star  \nu^{0}_{\iota} = \nu^{1}_{\epsilon}
\]
\end{lemma}
\begin{proof}
Let $\beta \in C_\alpha ({1})$ and $\gamma \in C_\alpha({0})$. Then
\begin{align*}
(\epsilon e^\beta + e^{\alpha + \beta}) ( \theta_\iota^\gamma e^\gamma + e^{\alpha + \gamma} )  &= (\epsilon \theta_\iota^\gamma b_{\beta,\gamma} + b_{\alpha+\beta, \alpha + \gamma}) e^{\beta + \gamma} \\
&\qquad + (\epsilon b_{\beta, \alpha + \gamma} + \theta^\gamma_\iota b_{\gamma, \alpha + \beta}) e^{\alpha + \beta + \gamma} \\
&= (\theta_\iota^\gamma + \epsilon) b_{\beta, \gamma} ( \epsilon e^{\beta + \gamma} + e^{\alpha + \beta + \gamma} )
\end{align*}
by our assumptions on $b_{\beta, \gamma}$.  Since $\beta + \gamma \in C_\alpha({1})$, the above product is in $A_{\nu^{1}_\epsilon}$.
\end{proof}

\begin{proposition}
Let $C =\bigoplus_{i=1}^r C_i$ be the direct sum of even weight codes $C_i$ all of length $m$, $n = m^r$ such that $n \geq 5$ and $m \geq 3$.  Then, $A_C$ is a $\mathbb{Z}_2$-graded axial algebra with
\begin{align*}
A_+ &= A_1 \oplus A_0 \oplus A_\lambda \oplus A_{\lambda - \frac{1}{2}} \oplus A_{\nu^{0}_+} \oplus A_{\nu^{0}_-}  \\
A_- &= A_{\nu^{1}_+} \oplus A_{\nu^{1}_-}
\end{align*}
and fusion law given by Table $\ref{tab:al2}$.
\begin{table}[!htb]
\setlength{\tabcolsep}{5pt}
\renewcommand{\arraystretch}{1.5}
\centering
\begin{tabular}{c|ccccccccc}
 & $1$ & $0$ & $\lambda$ & $\lambda-\frac{1}{2}$ & $\nu^{1}_+$ & $\nu^{1}_-$ & $\nu^{0}_+$ & $\nu^{0}_-$ \\ \hline
$1$ & $1$ &  & $\lambda$ & $\lambda-\frac{1}{2}$ & $\nu^{1}_+$ & $\nu^{1}_-$ & $\nu^{0}_+$ & $\nu^{0}_-$  \\
$0$ &  & $0$  & &    & $\nu^{1}_+$ & $\nu^{1}_-$ & $\nu^{0}_+$ & $\nu^{0}_-$ \\
$\lambda$ & $\lambda$ & & $1, \lambda - \frac{1}{2}$ & & $\nu^{1}_-$  & $\nu^{1}_+$   \\ 
$\lambda-\frac{1}{2}$ & $\lambda-\frac{1}{2}$& &   & $1, \lambda-\frac{1}{2}$ & $\nu^{1}_+$ & $\nu^{1}_-$ & $\nu^{0}_+, \nu^{0}_-$ &  $\nu^{0}_+, \nu^{0}_-$ \\
$\nu^{1}_+$ & $\nu^{1}_+$ & $\nu^{1}_+$  & $\nu^{1}_-$  &  $\nu^{1}_+$ & $X, \lambda$ & $X, \lambda$ &  $\nu^{1}_+, \nu^{1}_- $ &  $\nu^{1}_+, \nu^{1}_-$ \\
$\nu^{1}_-$ & $\nu^{1}_-$ & $\nu^{1}_-$  &  $\nu^{1}_+$ &  $\nu^{1}_-$  & $X, \lambda$  & $X, \lambda$ &  $\nu^{1}_+, \nu^{1}_-$ &  $\nu^{1}_+, \nu^{1}_-$ \\
$\nu^{0}_+$ & $\nu^{0}_+$ & $\nu^{0}_+$  & & $\nu^{0}_+, \nu^{0}_-$ & $\nu^{1}_+, \nu^{1}_-$ & $\nu^{1}_+, \nu^{1}_-$ & $X$ & $X$ \\
$\nu^{0}_-$ & $\nu^{0}_-$ & $\nu^{0}_-$  &  & $\nu^{0}_+, \nu^{0}_-$ & $\nu^{1}_+, \nu^{1}_-$ & $\nu^{1}_+, \nu^{1}_-$  & $X$ & $X$
\end{tabular}
\vspace{5pt}\\
where $X = 1,0, \lambda-\frac{1}{2}, \nu^{0}_+, \nu^{0}_-$
\caption{Fusion law for $\vert \alpha \vert = 2$}
\label{tab:al2}
\end{table}
\end{proposition}

\begin{proof}
The fusion law is the same as for the general case given in Table \ref{tab:small} except for the following entries.  By Lemma \ref{lambda*lambda}, we have $\lambda\star\lambda = 1, \lambda - \frac{1}{2}$, by Lemma \ref{lambda*nu}, we have $\lambda \star \nu^p_\pm$ and by Lemma \ref{lambda-1/2*nu} we have $\lambda-\frac{1}{2}\star \nu^{1}_\pm$.  By Lemmas \ref{al2N2}, \ref{al2nu0*nu0} and \ref{al2nu1*nu1}, we have the values for the $\nu^p_\pm$.  Once we have the table, it is easy to observe the grading.
\end{proof}

If in addition we make some assumptions about the structure parameters, we get a stronger result.

\begin{proposition}\label{Z2xZ2}
Let $C =\bigoplus_{i=1}^r C_i$ be the direct sum of even weight codes $C_i$ all of length $m$, $n = m^r$ such that $n \geq 5$ and $m \geq 3$.  Let $S$ be the set of conjugates of a weight two codeword $\alpha$ and suppose that $b_{\beta, \gamma} = b_{\alpha_i+\beta, \gamma}$ and $c_\beta = c_{\alpha_i+\beta}$ for all $\beta \in C_{\alpha_i}(1)$, $\alpha_i \in S$ and $\gamma \in C^*\setminus \{\alpha, \alpha^c\}$.  Then, the axial algebra $A_C$ has fusion law given by Table $\ref{tab:al2special}$ and has a $\mathbb{Z}_2 \times \mathbb{Z}_2$-grading given by
\begin{align*}
A_{(0,0)} &= A_1 \oplus A_0 \oplus A_{\lambda - \frac{1}{2}} \oplus A_{\nu^{0}_+} \oplus A_{\nu^{0}_-}  \\
A_{(1,0)} &= A_{\nu^{1}_+}  \\
A_{(0,1)} &= A_{\nu^{1}_-}  \\
A_{(1,1)} &= A_{\lambda}
\end{align*} 

\begin{table}[!htb]
\setlength{\tabcolsep}{7pt}
\renewcommand{\arraystretch}{1.5}
\centering
\begin{tabular}{c|ccccccccc}
 & $1$ & $0$ & $\lambda$ & $\lambda-\frac{1}{2}$ & $\nu^{1}_+$ & $\nu^{1}_-$ & $\nu^{0}_+$ & $\nu^{0}_-$ \\ \hline
$1$ & $1$ &  & $\lambda$ & $\lambda-\frac{1}{2}$ & $\nu^{1}_+$ & $\nu^{1}_-$ & $\nu^{0}_+$ & $\nu^{0}_-$  \\
$0$ &  & $0$  & &    & $\nu^{1}_+$ & $\nu^{1}_-$ & $\nu^{0}_+$ & $\nu^{0}_-$ \\
$\lambda$ & $\lambda$ & & $1, \lambda - \frac{1}{2}$ & & $\nu^{1}_-$  & $\nu^{1}_+$   \\ 
$\lambda-\frac{1}{2}$ & $\lambda-\frac{1}{2}$& &   & $1, \lambda-\frac{1}{2}$ & $\nu^{1}_+$ & $\nu^{1}_-$ & $\nu^{0}_+, \nu^{0}_-$ &  $\nu^{0}_+, \nu^{0}_-$ \\
$\nu^{1}_+$ & $\nu^{1}_+$ & $\nu^{1}_+$  & $\nu^{1}_-$  &  $\nu^{1}_+$ & $X$ & $\lambda$ &  $\nu^{1}_+ $ &  $\nu^{1}_+$ \\
$\nu^{1}_-$ & $\nu^{1}_-$ & $\nu^{1}_-$  &  $\nu^{1}_+$ &  $\nu^{1}_-$  & $\lambda$  & $X$ &  $\nu^{1}_- $ &  $\nu^{1}_- $ \\
$\nu^{0}_+$ & $\nu^{0}_+$ & $\nu^{0}_+$  & & $\nu^{0}_+, \nu^{0}_-$ & $ \nu^{1}_+$ & $ \nu^{1}_- $ & $X$ & $X$ \\
$\nu^{0}_-$ & $\nu^{0}_-$ & $\nu^{0}_-$  &  & $\nu^{0}_+, \nu^{0}_-$ & $ \nu^{1}_+$ & $\nu^{1}_- $  & $X$ & $X$
\end{tabular}
\vspace{5pt}\\
where $X = 1,0, \lambda - \frac{1}{2}, \nu^{0}_+, \nu^{0}_-$
\caption{Fusion law for $\vert \alpha \vert = 2$}
\label{tab:al2special}
\end{table}
\end{proposition}
\begin{proof}
The table is the same as Table \ref{tab:al2}, except for $\nu^1_\epsilon \star \nu^1_\iota$ and $\nu^1_\epsilon \star \nu^0_\iota$ which follow from the special cases of Lemmas \ref{al2nu1*nu1} and \ref{al2nu1*nu0}.
\end{proof}

It remains to consider the degenerate case where $n=m=3,4$ and there is only one weight partition $(1,1)$.  Here, the code algebra is an axial algebra its fusion law given by Tables \ref{tab:al2} or \ref{tab:al2special}, depending on the structure parameters, where the $\nu^0_\pm$ are ignored.  We observe that such a fusion law is still $\mathbb{Z}_2$- or $\mathbb{Z}_2 \times \mathbb{Z}_2$-graded depending on the structure parameters.


\section{$\mathbb{Z}_2$-grading}\label{sec:grading}

In this section we examine the fusion laws of the code algebras which are axial algebras more carefully.  We will classify when the fusion law for the small idempotents is $\mathbb{Z}_2$-graded.

\begin{theorem}\label{Z2grading}
We assume the assumptions of Theorem $\ref{codeisaxialalg}$.  Then the axial algebra $A_C$ has a $\mathbb{Z}_2$-graded fusion law if and only if it is one of the following
\begin{enumerate}
\item[$1.$] $C = \mathbb{F}_2^n$, $|\alpha| = 1$ and
\begin{enumerate}
\item[\textup{(}a\textup{)}] $n = 2$, $a=-1$.
\item[\textup{(}b\textup{)}] $n=3$.
\end{enumerate}

\item[$2.$] $C = \bigoplus C_i$ is the direct sum of even weight codes all of the same length $m$, $n \geq 3$, $|\alpha| = 2$.

\item[$3.$] $|\alpha| > 2$ where $D = \proj_\alpha(C)$ is a projective code, $\1 \in D$ and $D$ has a codimension one linear subcode $D_+$, with $ \1 \in D_+$, which is the union of weight sets of $D$.

In this case, we have
\begin{align*}
A_+ &= A_1 \oplus A_0 \oplus A_\lambda \oplus A_{\lambda - \frac{1}{2}} \oplus
\bigoplus_{m \in wt(D_+)} A_{\nu^{(m, |\alpha|- m)}_\pm} \\
A_- &= \bigoplus_{m \in wt(D) - wt(D_+)} A_{\nu^{(m, |\alpha|-m)}_\pm}
\end{align*}
\end{enumerate}
Moreover, the examples occurring in parts $(1)$ and $(2)$ are precisely those given in Sections $\ref{al=1}$ and $\ref{al=2}$.  For $|\alpha|=2$, the example in Section $\ref{al=2}$ is $\mathbb{Z}_2 \times \mathbb{Z}_2$-graded if additional assumptions are made on the structure parameters.
\end{theorem}

The restrictions on the code in the third case are fairly mild.  Indeed, it is not difficult to see that if $D$ has even length and  contains any odd codewords, then the even weight codewords of $D$ form a linear subcode $D_+$ of codimension one and $\1 \in D_+$.  Other examples with $D$ even also exist.  It remains then to extend $D$ to a code $C$ such that conjugates of $\1_D \in D$ in $C$ generate $C$ and check that $C$ is projective.

\begin{example}
Consider the code $C$ with generating matrix
\[ \left( \begin{tabular}{ccccc}
1 & 0 & 0 & 0 & 1  \\
0 & 1 & 0 & 0 & 1 \\
0 & 0 & 1 & 0 & 1 \\
0 & 0 & 0 & 1 & 1 
\end{tabular} \right) \]
Since $C$ is the even weight code, it is projective and $\Aut(C) \cong S_5$.  So conjugates of $\alpha := (1,1,1,1,0)$ generate $C$.  Then $D = \proj_\alpha(C) \cong \mathbb{F}_2^4$ and $D_+$ is the codimension one subcode of all even weight codewords in $D$.  This satisfies the conditions given in Case $3$ of Theorem \ref{Z2grading}.
\end{example}

We will prove the theorem via a series of lemmas.   We will deduce what the necessary conditions on the code are and then show that these examples are indeed $\mathbb{Z}_2$-graded.  Suppose that the fusion law $\mathcal{F}$ for a small idempotent in $A$ is $\mathbb{Z}_2$-graded with partition $\mathcal{F} = \mathcal{F}_+  \sqcup \mathcal{F}_-$. 

\begin{lemma}\label{finf*f}
Let $f \in \mathcal{F}$.  If $f \in f \star f$, then $f \in \mathcal{F}_+$.
\end{lemma}
\begin{proof}
If $f$ were in $\mathcal{F}_-$, then $f \in f \star f \in \mathcal{F}_+$, a contradiction.
\end{proof}

\begin{corollary}\label{10lambdagrading}
We have
\begin{enumerate}
\item[$1.$] $1,0 \in \mathcal{F}_+$
\item[$2.$] if $|\alpha| >2$ then $\lambda \in  \mathcal{F}_+$
\item[$3.$] $\lambda - \frac{1}{2} \in \mathcal{F}_+$, except possibly when $|\alpha| = 1$ and $a = -1$
\end{enumerate}
\end{corollary}
\begin{proof}
Part one follows from Lemma \ref{0*0} and the fact that this is the fusion law of an idempotent.  The second part follows from Lemma \ref{lambda*lambda}.  By Lemma \ref{lambda-1/2*lambda-1/2}, $\lambda - \frac{1}{2}$ is in $\mathcal{F}_+$ unless $a = -\frac{1}{|\alpha|}$.  However, provided $|\alpha| \neq 1$, then the eigenvalue $\lambda$ exists. Now, since $1, \lambda - \frac{1}{2} \in \lambda \star \lambda$ by Lemma \ref{lambda*lambda}, $\lambda-\frac{1}{2}$ must have the same grading as $1$, which is in the positive part.  Hence, $\lambda - \frac{1}{2} \in \mathcal{F}_+$, except possibly when $|\alpha| = 1$ and $a = -1$.
\end{proof}

To complete the grading for $\lambda - \frac{1}{2}$ we consider the one case remaining from above.  Note that, since $A_C$ is assumed to be non-degenerate, $n \geq 2$. This is a somewhat fiddly calculation.

\begin{lemma}\label{lambda-1/2grading}
If $|\alpha| \neq 1$, or $n \neq 2$, or $a \neq -1$, then $\lambda - \frac{1}{2} \in \mathcal{F}_+$.
\end{lemma}
\begin{proof}
For a contradiction, suppose that $\lambda - \frac{1}{2} \in \mathcal{F}_-$, so $|\alpha| = 1$ and $a=-1$ , but $n \neq 2$.  By assumption, the set of conjugates of $\alpha$ generate the code, so $C$ must be the whole code $\mathbb{F}_2^n$.  Since $n  \geq 3$, $C$ has a weight partition of $\alpha$ which is $p = (0,1)$.

Let $\beta \in C'_\alpha(p)$; with loss of generality, we may assume that $\alpha \cap \beta = \emptyset$.  Considering $\lambda-\frac{1}{2} \star \nu^p_\epsilon$ we have
\[
(2\mu c_\alpha t_\alpha - e^\alpha)(\theta^\beta_\epsilon e^\beta + e^{\alpha+\beta}) = - b_{\alpha, \beta} e^\beta + (-2\mu c_\alpha - b_{\alpha, \beta} \theta^\beta_\epsilon)e^{\alpha+\beta}
\]
This is contained in $A_{\nu^p_+} \oplus A_{\nu^p_-}$ and, since $b_{\alpha, \beta} \neq 0$, it is clear that it is not zero.  By assumption, $\lambda - \frac{1}{2} \in \mathcal{F}_-$.  Hence, to preserve the grading, $\nu^p_+$ and $\nu^p_-$ must have different gradings and $\lambda-\frac{1}{2} \star \nu^p_\epsilon = \nu^p_{-\epsilon}$.

Now, consider $\nu^p_\epsilon \star \nu^p_\epsilon$.  By Lemma \ref{nu*nucalc},
\begin{align*}
 w^\beta_\epsilon w^\beta_\epsilon &= (\theta^\beta_\epsilon)^2 c_\beta t_\beta + c_{\alpha+\beta} t_{\alpha+\beta} + 2\theta^\beta_\epsilon b_{\beta, \alpha+\gamma} e^\alpha \\
 &= \left( (\theta^\beta_\epsilon)^2 c_\beta + c_{\alpha+\beta} \right) t_\beta + c_{\alpha+\beta} t_\alpha + 2\theta^\beta_\epsilon b_{\beta, \alpha+\gamma} e^\alpha
\end{align*}
By the grading, this must be positive, so the above must lie in $A_1 \oplus A_0$.  Hence, for some $x \in \mathbb{F}^\times$,
\begin{align*}
\lambda x &= c_{\alpha+\beta} \\
\mu x &= 2\theta^\beta_\epsilon b_{\beta, \alpha+\gamma}
\end{align*}
Eliminating $x$, we find that
\[
\theta^\beta_\epsilon = \frac{\mu c_{\alpha  + \beta}}{2\lambda b_{\beta, \alpha+ \gamma}}
\]
However, this must hold for both $\epsilon = -1, +1$, contradicting the fact that $\theta^\beta_+ \neq \theta^\beta_-$.  Hence, $\lambda-\frac{1}{2}$ must be in the positive part.
\end{proof}

We consider the grading of $\nu^p_\pm$ for the weight partitions $p$ of $\alpha$.

\begin{lemma}\label{nugrading}
If $\xi_\beta \neq 0$, then the eigenspaces $\nu^p_+$ and $\nu^p_-$ have the same grading.
\end{lemma}
\begin{proof}
Since we assume that $p$ is a weight partition, $n \geq 3$.  Hence, by Lemma \ref{lambda-1/2grading}, $\lambda - \frac{1}{2} \in \mathcal{F}_+$.  Let $\beta \in C_\alpha(p)$.   By Lemmas \ref{nu*nucalc} and \ref{coefsubs}, 
\[
w^\beta_+ w^\beta_- = -c_\beta t_\beta + c_{\alpha+\beta} t_{\alpha+\beta} -2 \xi_\beta b_{\beta, \alpha+\beta} e^\alpha
\]
Since $b_{\beta, \alpha+\beta} \neq 0$, the coefficient of $e^\alpha$ in the above is non-zero if and only if $\xi_\beta \neq 0$.  However, $e^\alpha$ is in $A_1 \oplus A_{\lambda-\frac{1}{2}}$ and, by Lemmas \ref{10lambdagrading} and our assumptions, this is in $A_+$.  Hence the above product is always in $A_+$ and the grading of $\nu^p_+$ and $\nu^p_-$ is the same.
\end{proof}

\begin{lemma}\label{nusame}
The eigenspaces $\nu^p_+$ and $\nu^p_-$ have the same grading, except possibly when $|\alpha| = 2$ and $p = (1,1)$.
\end{lemma}
\begin{proof}
By Lemma \ref{nugrading}, we only need to consider the case where $\xi_\beta=0$.  This case does not occur when $|\alpha| = 1$ and we assume $|\alpha| \neq 2$, so we may consider $|\alpha| >2$.  Here, $\lambda$ is an eigenvalue and so, by Lemma \ref{lambda*nu}, $\lambda \star \nu^p_\pm = \nu^p_\mp$.  Since $\lambda \in \mathcal{F}_+$, this implies that they have the same grading. Now observe that, for $|\alpha| = 2$, $\xi_\beta=0$ implies $p = (1,1)$.
\end{proof}

From the above results, we have that $1,0,\lambda$ are all in $\mathcal{F}_+$, $\lambda -\frac{1}{2}$ is also in $ \mathcal{F}_+$ if $|\alpha| \neq 1$, and $\nu^p_+$ and $\nu^p_-$ have the same grading unless $| \alpha| = 2$.  This suggests the following split into cases:

\begin{enumerate}
\item $|\alpha| = 1$
\item $|\alpha| = 2$
\item $|\alpha| > 2$
\end{enumerate}

We now give some lemmas which will help determine the grading of the weight partition $(0, |\alpha|)$.

\begin{lemma}\label{qexists}
Suppose that $|\alpha| \neq 1$ and $p = ( 0, |\alpha|) \in P_\alpha$ is a weight partition of $\alpha$.  Then there exists a weight partition $q \neq p$.
\end{lemma}
\begin{proof}
Since $C$ is projective and $|\alpha| \neq 1$, there exists some $\beta \in C$ such that $\alpha \cap \beta \neq \0, \alpha$.  Hence, there exists some weight partition $q = p(\beta)$ not equal to $p = ( 0, |\alpha|)$.
\end{proof}

\begin{lemma}\label{wwneq0}
Suppose that $p,q \in P_\alpha$ are weight partitions of $\alpha$ and let $\beta \in C_\alpha(p)$, $\gamma \in C_\alpha(q)$ with $\gamma \neq \beta, \beta^c, \alpha+\beta, \alpha+\beta^c$.  If $w^\beta_\epsilon w^\gamma_\iota = 0$, then
\[
w^\beta_{-\epsilon} w^\gamma_\iota \neq 0 \neq w^\beta_\epsilon w^\gamma_{-\iota}
\]
\end{lemma}
\begin{proof}
If $w^\beta_\epsilon w^\gamma_\iota = 0$, then by Lemma \ref{nu*nucalc}, $\theta^\beta_\epsilon \theta^\gamma_\iota b_{\beta, \gamma} + b_{\alpha+\beta, \alpha+\gamma} = 0$.  For $w^\beta_{-\epsilon} w^\gamma_\iota$ to equal zero, we would also require $\theta^\beta_{-\epsilon} \theta^\gamma_\iota b_{\beta, \gamma} + b_{\alpha+\beta, \alpha+\gamma} = 0$ and hence $\theta^\beta_\epsilon = \theta^\beta_{-\epsilon}$, a contradiction.  Similarly $w^\beta_\epsilon w^\gamma_{-\iota} \neq 0$.
\end{proof}

\begin{lemma}\label{nu0}
Suppose that $|\alpha| \neq 1,2$ and $p = ( 0, |\alpha|) \in P_\alpha$ is a weight partition of $\alpha$.  Then $\nu^p_+$ and $\nu^p_-$ are in $\mathcal{F}_+$.
\end{lemma}
\begin{proof}
By Lemma \ref{qexists}, there exists another weight partition $q \neq p$.  Let $\beta \in C_\alpha(p)$, $\gamma \in C_\alpha(q)$.  Since $q \neq p$, $\gamma \neq \beta, \beta^c, \alpha+\beta, \alpha+\beta^c$, so by Lemma \ref{wwneq0}, there exists $\epsilon, \iota = \pm1$ such that $w^\beta_\epsilon w^\gamma_\iota \neq 0$.  However, it is clear that $p(\beta+\gamma) = p(\gamma) = q$.  So, $\nu^q_\delta \in \nu^p_\epsilon \nu^q_\iota$.  Since $|\alpha| \neq 2$, by Lemma \ref{nusame}, $\nu^q_\pm$ have the same grading and therefore $\nu^p_\pm \in \mathcal{F}_+$.
\end{proof}

\subsection{$|\alpha|=1$}

When $|\alpha|=1$, we do not have the eigenvalue $\lambda$.  Also, it is clear that the only possible weight partition of $\alpha$ is $0 := (0,1)$  and this exists provided $n \geq 3$.  As noted previously, since the conjugates of $\alpha$ generate $C$, the code is the whole code $\mathbb{F}_2^n$, for some $n$.

Firstly, $n=1$ leads to a degenerate code algebra and so can be disregarded. Secondly, suppose that $n = 2$.  Then, $1,0 \in \mathcal{F}_+$ and the only other eigenvalue is $\lambda - \frac{1}{2}$.  By Lemma \ref{lambda-1/2grading}, this can only be in $\mathcal{F}_-$ when $a=-1$.  It is easy to see that the fusion law in this case is indeed $\mathbb{Z}_2$-graded and it is described in Section \ref{al=1}.

Thirdly, assume that $n \geq 3$.  Then, $1,0,\lambda-\frac{1}{2} \in \mathcal{F}_+$ and $\nu^0_+$ and $\nu^{0}_-$ have the same grading, so the only possible $\mathbb{Z}_2$-grading if where $\nu^{0}_+$ and $\nu^{0}_-$ are both in $\mathcal{F}_-$.  This case is described in the example in Section \ref{al=1} and a $\mathbb{Z}_2$-grading is only possible if $n = 3$.

\subsection{$|\alpha| =2$}

Now suppose that $|\alpha| =2$.  Recall that an indecomposable code is one which is not equivalent to the direct sum of two other codes. In other words, for all generating matrices $G$ and permutation matrices $P$, $GP$ is not a block diagonal matrix.

\begin{lemma}\label{evenweightC}
Let $C$ be an indecomposable linear code which is generated by weight two codewords.  Then, $C$ is the even weight code, which consists of all even codewords.
\end{lemma}
\begin{proof}
We show this by induction on the length $n$.  Clearly it is true for length $2$.  So let $C$ be length $n$ and dimension $k$.  We define a code $C'$ from $C$ by removing all codewords with a $1$ in the last position and then puncturing the code in the last position.  So, $C'$ has length $n-1$ and is dimension $k-1$.  

We claim that $C'$ is indecomposable.  Suppose not, then there exists a generating matrix $M'$ of $C'$ and a permutation matrix $P$ such that $M'P$ is a block diagonal matrix.  Since $C$ is generated by weight two codewords there exists $\alpha \in C$ of weight two with a one in the last position.  Hence the matrix formed from $M'$ by adding a column of zeroes and the adjoining $\alpha$ has rank $k$ and so generates $C$.  However, by permutation of the columns it is of block diagonal form, so $C$ is decomposable, a contradiction.

Since $C'$ is indecomposable and generated by weight two elements, by induction, it is the even weight code.  In particular, it has dimension $n-2$.  Hence $C$ is an even code of dimension $n-1$ and so is the even weight code.
\end{proof}

\begin{corollary}\label{al=2code}
Let $|\alpha|=2$.  Then, $A_C$ is the code algebra of a code $C$ which is a direct sum of indecomposable even weight codes all of length $m \geq 3$.
\end{corollary}
\begin{proof}
By Lemma \ref{evenweightC}, $C$ is a direct sum of codes of even weight.  Since all the codewords of $S$ are conjugate under the automorphism of the code, the length of each indecomposable subcode must be the same.  In particular the length cannot be $2$ as then the code would not be projective.
\end{proof}

Note that if $C$ is the even code of length $3$ or $4$, then there is just one weight partition $1 := (1,1)$.  In all other cases, there are two possible weight partitions, $0 := (0,2)$ and $1 = (1,1)$.

\begin{lemma}
Suppose that $n \geq 5$ and so $0 = (0, 2) \in P_\alpha$ is a weight partition of $\alpha$.  Then $\nu^0_+$ and $\nu^0_-$ are in $\mathcal{F}_+$.
\end{lemma}
\begin{proof}
By Corollary \ref{al=2code}, $C = \bigoplus C_i$, where $c_i$ are indecomposable even weight codes of length $m \geq 3$. Since $n \geq 5$, there exists $\beta, \gamma \in C_\alpha(0)$ such that $|\alpha \cap \beta| = |\alpha \cap \gamma| =0$, $|\beta| = |\gamma| = 2$ and $|\beta \cap \gamma| = 1$.  In particular, this implies that $\gamma \neq \beta, \beta^c, \alpha+\beta, \alpha+\beta^c$.  So, by Lemma \ref{nu*nucalc},
\[
w^\beta_+ w^\gamma_- = ( \theta^\beta_+ \theta^\gamma_- b_{\beta,\gamma} + b_{\alpha+\beta, \alpha+\gamma})e^{\beta+\gamma} + (\theta^\beta_+ b_{\beta, \alpha + \gamma} + \theta^\gamma_- b_{\alpha+\beta, \gamma})e^{\alpha+\beta+\gamma}
\]
Since $|\beta| = |\gamma| = 2$ and $|\beta \cap \gamma| = 1$, $\alpha+\beta \in C_\gamma(1)$ and $\alpha+\gamma \in C_\beta(1)$, so by our assumptions on the $b$ structure parameters, $b_{\beta, \alpha + \gamma} = b_{\alpha+\beta, \gamma}$.  Also, as $|\alpha \cap \beta| = |\alpha \cap \gamma|$, $\theta^\gamma_- = \theta^\beta_-$, so the coefficient of $e^{\alpha+\beta+\gamma}$ above is
\[
b_{\beta, \alpha + \gamma}(\theta^\beta_+ + \theta^\beta_-) = -2 b_{\beta, \alpha + \gamma} \xi_\beta
\]
Since $p(\beta) = 0$, the above is non-zero and so we have $0 \neq \nu^0_+ \star \nu^0_- \in \nu^0_\pm$ and hence $\nu^0_+$ and $\nu^0_-$ are in $\mathcal{F}_+$.
\end{proof}

By Lemmas \ref{10lambdagrading}, \ref{lambda-1/2grading} and \ref{nu0}, the eigenvalues $1$, $0$, $\lambda - \frac{1}{2}$ and, where they exist, $\nu^0_+$ and $\nu^0_-$ are all in $\mathcal{F}_+$, so the only possible members of $\mathcal{F}_-$ are $\lambda$, $\nu^1_+$ and $\nu^1_-$. We see in the example from Section \ref{al=2} that in general, $\lambda \in \mathcal{F}_+$ and $\nu^1_+, \nu^1_- \in \mathcal{F}_-$.  However, if we make additional assumptions on the structure parameters, then we have a $\mathbb{Z}_2\times \mathbb{Z}_2$-grading, where the $\lambda$, $\nu^1_+$ and $\nu^1_-$ represent the three involutions in $\mathbb{Z}_2\times \mathbb{Z}_2$.

\subsection{$|\alpha| > 2$}

From now on we will assume that $|\alpha| > 2$.  Hence, $1, 0, \lambda, \lambda - \frac{1}{2}$ are all in $\mathcal{F}_+$.  By Lemma \ref{nusame}, we also have that the grading of $\nu^p_+$ is the same as that of $\nu^p_-$ for all $p \in P_\alpha$.  So we just need to determine the grading on the $\nu^p_\pm$.

We claim that the grading from the $\nu^p_\pm$ eigenspaces induces a grading of the code $C$.  That is, we can define a map $\gr\colon C \to \mathbb{Z}_2$ by
\begin{align*}
\beta &\mapsto \gr(w^\beta_\pm) \\
\0,\1,\alpha, \alpha^c &\mapsto 1
\end{align*}
for $\beta \in C^* \setminus \{ \alpha, \alpha^c \}$, where $\gr(w^\beta_\pm)$ denotes the grading in the algebra of $w^\beta_\pm$.

\begin{lemma}
Viewing $C$ as an additive group, $\gr\colon C \to \mathbb{Z}_2$ is a homomorphism of groups.
\end{lemma}
\begin{proof}
First note that, by Lemma \ref{nusame}, the grading of $\nu^p_+$ is the same as that of $\nu^p_-$ for all $p \in P_\alpha$.  Hence the map is well-defined. For $\beta, \gamma \in C^* \setminus \{ \alpha, \alpha^c \}$ and $\gamma \neq \beta, \beta^c, \alpha+\beta, \alpha+\beta^c$, the grading of the code follows from that of the algebra.

So, suppose that $\gamma  = \beta, \beta^c, \alpha+\beta, \alpha+\beta^c$, then $p(\beta) = p(\gamma)$ and the product $w^\beta_\pm w^\gamma_\pm$ is in $A_1 \oplus A_0 \oplus A_\lambda \oplus A_{\lambda - \frac{1}{2}} \subset A_+$.  After checking the remaining cases, we see that $\gr$ is a homomorphism.
\end{proof}

We denote by $C_+$ and $C_-$ the positively and negatively graded parts of $C$, respectively.  Note that, since $\gr$ is a homomorphism, the kernel, which is $C_+$, has the same size as $C_-$.

Let $D = \proj_\alpha(C)$.  Since $C$ is projective, $D$ is too and $\proj_\alpha(\alpha)$ is the $\1 \in D$.

\begin{lemma}
We have $\gr(\ker(\proj_\alpha)) = 1$.
\end{lemma}
\begin{proof}
The kernel of the projection is
\[
\ker(\proj_\alpha) = \{ \beta : \alpha \cap \beta = \emptyset \}
\]
which is contained in $C_\alpha((0, |\alpha|)) \cup \{0, \alpha^c\}$.  By Lemma \ref{nu0} and the definition of the grading map, this is all in $C_+$.
\end{proof}

\begin{corollary}
The projection map induces a non-trivial grading on $D$.
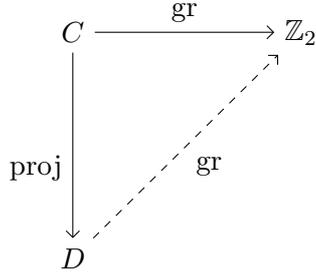
\begin{center}
\begin{tikzpicture}[scale=3]
\node (1) at (0,1) {$C$};
\node (2) at (1,1) {$\mathbb{Z}_2$};
\node (3) at (0,0) {$D$};
\path[->, >=angle 90]

(1) edge node[anchor = south]{$\gr$} (2)
(1) edge node[anchor = north east]{$\proj$} (3)
(3) edge[dashed] node[anchor = north west]{$\gr$} (2);
\end{tikzpicture}
\end{center}
\end{corollary}

Note that a weight partition $p = ( m, |\alpha|-m)$ of $\alpha$ corresponds to a union of two weight sets of $D$, namely the set of all codewords of $D$ of weights $m$, or $|\alpha|-m$.  Hence, $D_+$ is a union of weight spaces of $D$ and it is closed under taking complements, so $ \1 \in D_+$.  Since it is also closed under addition and $|D_+| = |D_-|$, it is also a codimension one subcode of $D$.

Conversely, if we have a code with the required properties, then it is clear that it induces a grading on the fusion law. This completes the proof of Theorem \ref{Z2grading}.

\end{document}